\documentclass[a4paper,11pt]{article}

\usepackage[all]{xy}
\usepackage[english]{babel}
\usepackage[latin1]{inputenc}
\usepackage{amsfonts}
\usepackage{amsthm}
\usepackage{amsmath}
\usepackage{amsfonts}
\usepackage{latexsym}
\usepackage{amssymb}
\usepackage{mathrsfs}
\usepackage[usenames]{color}
\usepackage{amsmath,graphicx}

\newtheorem{fed}{Definition}[section]
\newtheorem{teo}[fed]{Theorem}
\newtheorem{cor}[fed]{Corollary}
\newtheorem*{teo*}{Theorem}

\newtheorem{pro}[fed]{Proposition}

\newtheorem{defi}[fed]{Definition}
\theoremstyle{definition}
\newtheorem{rem}[fed]{Remark}

\newtheorem*{teoD}{Douglas' theorem}

\def\bdem{\begin{proof}}
\def\edem{\renewcommand{\qed}{\hfill $\blacksquare$}
\end{proof}}

\newcommand{\pint}[1]{\displaystyle \left \langle #1 \right\rangle}

\makeatletter
\newcommand{\makecircled}[2][\mathord]{#1{\mathpalette\make@circled{#2}}}
\newcommand{\make@circled}[2]{%
  \begingroup\m@th
  \vphantom{\biggercirc{#1}}%
  \ooalign{$#1\biggercirc{#1}$\cr\hidewidth$#1#2$\hidewidth\cr}%
  \endgroup
}
\newcommand{\biggercirc}[1]{%
  \vcenter{\hbox{\scalebox{1.4}{$\m@th#1\bigcirc$}}}%
}
\makeatother

\def\cH{\mathcal{H}}

\def\eme{\mathcal{M}}

\def\cH{\mathcal{H}}

\date{}
\begin{document}

\title{Proper splittings of Hilbert space operators}
\author{Guillermina Fongi $^{a}$, M. Celeste Gonzalez $^{b}$ $^{c}$ \footnote{The author was supported in part by  04/E127 UNCoMA}\\ 
\fontsize {9}{9} \selectfont$^a$ 
Centro Franco Argentino de Ciencias de la Informaci\'on y de Sistemas, CIFASIS-CONICET \\ \selectfont \fontsize {9}{9} \selectfont Ocampo y Esmeralda (2000)  Rosario, Argentina.
\\
\fontsize {9}{9} \selectfont$^b$ Instituto Argentino de Matem\'atica ``Alberto P. Calderón'', IAM-CONICET\\
\fontsize {9}{9} \selectfont Saavedra 15, Piso 3 (1083), Buenos Aires, Argentina.\\
\fontsize {9}{9} \selectfont $^c$ Instituto de Ciencias, Universidad Nacional de General Sarmiento.\\
\fontsize {9}{9} \selectfont{$^a$  gfongi@conicet.gov.ar, $^b$ celeste.gonzalez@conicet.gov.ar }}

\date{}
\maketitle


\fontsize {10}{10} \selectfont {{\sl {Keywords: \ }}{splitting of operators, operator orders, operator factorization}

\begin{abstract}
Proper splittings of operators are commonly used to study the convergence of iterative processes. In order to approximate solutions of operator equations, in this article we deal with proper splittings of closed range bounded linear  operators  defined on  Hilbert spaces. We study the convergence of  general proper splittings of operators in the infinite dimensional context. We also propose some particular splittings for special classes of operators and we study different criteria of convergence and comparison for them. In some cases, these criteria are given under hypothesis of operator order relations. In addition, we relate these results with the concept of the symmetric approximation of a frame in a Hilbert space. 
\end{abstract}

\maketitle
\section{Introduction}

Given an invertible matrix $T\in M^{n}(\mathbb{C})$, a splitting of $T$ is a partition $T=U-V$, where $U,V\in  M^{n}(\mathbb{C})$ with $U$  invertible.
The theory of regular splittings for invertible matrices, i.e.,  a splitting $T=U-V$ such that $U^{-1}$ has all its entries nonnegative started in a work due to Varga \cite{varga}, where  an iterative process to get the unique solution of the system $Tx=w$ was given.  To be more precise, given $T\in M^{n}(\mathbb{C})$ invertible and  $T=U-V$ a regular splitting of $T$,  Varga defined the iterative process
$$
x^{i+1}=U^{-1}V x^i+U^{-1}w,
$$
and he proved that this process converges for every initial $x^0$ if and only if the spectral radius of $U^{-1}V$, $\rho(U^{-1}V)<1$, is less than 1. Moreover, in such a case,  the process converges to $T^{-1}w$.

The idea of Varga was extended by Berman and Plemmons \cite{MR348984} in order to define an iterative process that allows to get the minimum norm least square solution of a system $Tx=w$, for $T\in M^{m\times n}(\mathbb{C})$. For this purpose, the concept of proper splitting of $T\in M^{m\times n}$ was introduced: a proper splitting of $T$ is a decomposition $T=U-V$ where $\mathcal{R}(U)=\mathcal{R}(T)$ and $\mathcal{N}(U)=\mathcal{N}(V)$. Then, Berman and Plemmos proved that given a proper splitting of $T\in M^{m\times n}(\mathbb{C})$ the iterative process

\begin{equation}\label{BermannPlemmos metodo}
x^{i+1}=U^{\dagger}V x^i+U^{\dagger}w,
\end{equation}
converges for every initial $x^0$ if and only if $\rho(U^\dagger V)<1$ and, in such a case, it converges to $T^{\dagger}w$, the minimum norm least square solution of $Tx=w$. 

In the literature  different kinds of splittings and proper splittings of a matrix can be found. Some of them are  regular splittings \cite{varga}, \cite{MR745083}, nonnegative splittings with first and second type \cite{MR1113154}, \cite{MR1930390}, weak regular splittings with first and second type \cite{MR1286436}, \cite{MR1628383}, \cite{MR1969059}, proper splittings \cite{MR3141710}, \cite{MR348984}, \cite{AG-splitting}, $P$-regular splittings \cite{MR0075677}, \cite{John}, $P$-proper splittings \cite{MR3671533}, weak nonnegative splittings of the first and the second type \cite{MR1695402}, among others. For each class of splitting there exist several results that guarantee the convergence of the iterative process associated. Also there exist results that compare the speed of convergence of the iterative method for different splittings of a matrix.

 Arias and Gonzalez \cite{AG-splitting} used the concept of proper splitting to extend the iterative method (\ref{BermannPlemmos metodo}) and so to define an iterative process that converges to a reduced solution of the solvable matrix equation $TX=W$. There, criteria of convergence and  comparison for proper splittings were given under hypothesis of positivity (according to L\"owner order) of certain matrices. Also in \cite{AG-splitting} it was proposed a particular proper splitting which is defined for every matrix in $M^{m\times n}(\mathbb{C})$; namely, the polar proper splitting. There, the convergence of the polar proper splitting  was studied and  it was shown that this particular splitting is "better" than others on certain classes of matrices. 

Along this article we will be interested in  extending the treatment of proper splittings of matrices to the context of Hilbert space operators in order to approximate solutions of an operator equation $TX=W$. 
 On one hand we generalize to the infinite dimensional context  results of \cite{AG-splitting} like a criterion of convergence for general proper splittings. Also, we study the polar proper splitting of a closed range operator and we characterize its convergence. We introduce particular proper splittings for the classes of split and Hermitian operators, we analyze their  convergence and we establish different criteria of comparison. In addition, throughout the article we focus in  comparing proper splittings of operators which are related under different factorizations. In particular some of the factorizations considered are given by means of the star order, the minus order, the sharp order. Also different proper splittings of operators in classes as the product of orthogonal projections, the product of an orthogonal projection by an Hermitian operator and the product of an orthogonal projection by a positive operator, are investigated. 

The reader is referred to \cite{2001DjorStani,2009Nacevska,LiuHuang} among others, for results on splittings,  iterative methods and criteria of convergence for splittings on infinite dimensional spaces. For example,  Nacevsca \cite{2009Nacevska} studied iterative methods for computing generalized inverses of linear operators by means of the method of splitting of operators.  Liu and Huang \cite{LiuHuang} studied the convergence of an iterative method to find solutions of a linear system $Tx=w$ by means of a proper splitting of the operator $GT$, where $G$ is an operator with range and nullspace prescribed.

The paper is organized as follows. Section 2 contains  notations and preliminaries results. In particular, we recall definitions of some operator orders and generalized inverses that will be used along the paper. Also, an  extension to the infinite dimensional case of a characterization of the L\"owner order for positive operators due to Baksalary, Liski and Trenkler \cite{MR1032688,AG-splitting} is given.   
In Section 3 the study of general  proper splittings given in \cite{AG-splitting} for the  finite dimensional case is now generalized  for Hilbert space operators. Here the main result is Theorem \ref{condiciones suficientes de convergencia}, where we provide sufficient conditions for the convergence of proper splittings under certain compacity hypothesis.
Section 4 is devoted to study the polar proper splitting. We characterize its convergence in Theorem \ref{convergencia del polar}. In Theorem \ref{polar sp orden estrella} the polar proper splittings of two operators related by the star order,  are compared. Also we show that the star order condition can not be replaced for a sharp order nor a minus order condition. In Section 5 we present two particular proper splittings for a split operator, i.e, an operator $T$ which satisfies the condition $\mathcal{R}(T)\dot+\mathcal{N}(T)=\mathcal{H}$. The splittings for such a $T$ are defined by means of operators that emerge from that decomposition of $\mathcal{H}$; namely,  the $Q$-proper splitting and the group proper splitting. For each one of them we give conditions that guarantee its convergence. We also provide some criteria of convergence for the splittings of two split operators related with the star and the sharp order. In Section 6 we study particular proper splittings associated to a  closed range Hermitian operator $T$. Here we consider the MP-proper splitting and the projection proper splitting, defined by means of the Moore-Penrose operator of $T$ and the orthogonal projection onto the range of $T$, respectively. We characterize the convergence of these proper splittings and we stablish a comparison criterion between them and the polar proper splitting. On the other hand, we apply the projection proper splitting to induce a proper splitting of an operator that can be factorized in terms of some Hermitian operator (Theorem \ref{inducido en PLh} and its corollaries). In subsection 6.1 we illustrate how we can apply the theory of splittings operators to get a symmetric approximation of a given frame.

Finally,  given $S,T\in\mathcal{L}(\mathcal{H})$, in Section 7 we study the induced splittings by $T$ on $S$ in the particular case that $S$ and $T$ are related by an invertible operator.

\section{Preliminaries}

Along this article $\mathcal{H}$ is a complex Hilbert space with inner product $\pint{.,.}$ and $\mathcal{L}(\mathcal{H})$ is the algebra of bounded linear operators defined from $\mathcal{H}$ to $\mathcal{H}$. The norm of $T\in\mathcal{L}(\mathcal{H})$ is $\|T\|=\sup\{\|Tx\|: \ \|x\|=1\}$.   By $\mathcal{K}$, $\mathcal{L}^h$, $\mathcal{L}^+$,  $\mathcal{U}$, $\mathcal{Q}$ and $\mathcal{P}$ we denote the classes of compact operators, Hermitian operators, positive operators, unitary operators,  oblique projections and orthogonal projections of $\mathcal{L}(\mathcal{H})$, respectively.  Given $T\in\mathcal{L}(\mathcal{H})$, $T^*$ denotes the adjoint  operator of $T$, $\mathcal{R}(T)$ denotes the range of $T$ and $\mathcal{N}(T)$ denotes the nullspace of $T$.  The direct sum between subspaces is denoted by $\dot+$.  Given closed subspaces $\mathcal{S}, \mathcal{T}\subseteq \mathcal{H}$ such that $\mathcal{S}\dot+\mathcal{T}=\mathcal{H}$, $Q_{\mathcal{S}//\mathcal{T}}$ indicates the oblique projection onto $\mathcal{S}$ along $\mathcal{T}$.  In particular, the orthogonal projection onto $\mathcal{S}$ is denoted by $P_{\mathcal{S}}$. We  will also use the notation $P_T$ to indicate the orthogonal projection onto $\mathcal{R}(T)$, for a closed range operator $T\in\mathcal{L}(\mathcal{\mathcal H})$. The spectrum and the spectral radius of  $T\in\mathcal{L}(\mathcal{H})$ are denoted by  $\sigma(T)$ and $\rho(T)$, respectively. It is well known that $\rho(T)\leq \|T\|$ and, if $T$ is a normal operator (i.e., $T^*T=TT^*$) then $\rho(T)= \|T\|$.

 The classical L\"owner order for Hermitian operators is denoted by $\leq$. Given $S,T\in\mathcal{L}^h$, it holds that $S\leq T$ if and only if $0\leq T-S$, or equivalently $0\leq \pint{(T-S)x,x}$ for all $x\in\mathcal{H}$.

The following result on range inclusion and operator factorization is due to Douglas \cite{MR0203464} and \cite{arias2008generalized}:

\begin{teoD}
	Consider $S,T \in \mathcal{L}(\mathcal{H})$. The following conditions are equivalent:
	\begin{enumerate}
		\item $\mathcal{R}(S)\subseteq \mathcal{R}(T)$;
		\item there exists a number $\lambda >0$ such that $SS^*\leq \lambda TT^*$;
		\item there exists $R\in\mathcal{L}(\mathcal{H})$ such that $TR=S$.
	\end{enumerate}
	Moreover, if any of the above conditions holds and $\eme\subseteq \mathcal{H}$ is a closed subspace such that $\eme\dot+\mathcal{N}(T)=\mathcal{H}$ then there exists a unique operator $X_\eme\in\mathcal{L}(\mathcal{H})$ such that $TX_\eme=S$ and $\mathcal{R}(X_\eme)\subseteq \eme$.  The operator  $X_\eme$ is called the reduced solution for $\eme$ of the equation $TX=S$ and also satisfies that $\mathcal{N}(X_\eme)=\mathcal{N}(S)$.
	
		  In the case that $\eme=\mathcal{N}(T)^\bot$, the unique solution $X_r\in\mathcal{L}(\mathcal{H})$ such that $TX_r=S$ and $\mathcal{R}(X_r)\subseteq \mathcal{N}(T)^\bot$ is called {\it{the Douglas reduced solution}} of the equation $TX=S$. In addition $X_r$ satisfies that $\mathcal{N}(X_r)=\mathcal{N}(S)$ and $\|X_r\|=inf \{\lambda: \  SS^*\leq \lambda TT^*\}$. 
\end{teoD}

Remember that a general inverse (or pseundoinverse) of $T\in\mathcal{L}(\mathcal H)$ is an operator which satisfies the equations $TXT=T$ and $XTX=X$. It is well known that a general inverse of $T\in\mathcal{L}(\mathcal{H})$ is bounded if and only if $T$ has closed range.  There exist different kinds of pseudoinverses of an operator according to the restriction imposed. In this article we deal with the Moore-Penrose inverse and with the group inverse of an operator. Next we collect some results about them.

Given $T\in\mathcal{L}(\mathcal{H})$ with closed range, $T^\dagger\in\mathcal{L}(\mathcal H)$ denotes the {\it{Moore Penrose inverse}} of $T$. The following two characterizations for the Moore-Penrose inverse of a closed range operator are known and can be found in  \cite{Gro}, \cite{arias2008generalized}: 
\begin{enumerate}
\item $T^ \dagger$ is the unique operator that satisfies simultaneously the following four equations:
$$
TXT=T,  \ \  \ \  XTX=X, \ \  \ \ TX=P_T, \ \ \ \  XT=P_{T^*};
$$

\item $T^\dagger$ is the Douglas reduced solution of the equation $TX=P_T$. 
\end{enumerate}

If $S,T\in\mathcal{L}(\mathcal{H})$ are invertible operators it is well known that $(ST)^{-1}=T^{-1}S^{-1}$. The extension of this identity to non invertible operators is known as ``the reverse order law'' for the Moore-Penrose inverse; i.e., $(ST)^\dagger =T^\dagger S^\dagger$. However this equality does not hold in general.  A  classical result due to Greville \cite{greville1966note} provides a  characterization of the reverse order law for matrices. Next, we state its extension to bounded linear operators in Hilbert spaces proved in \cite{MR317088, MR651705}. 

\begin{pro}\label{dagger de AB con rg iguales} Consider $S,T\in\mathcal{L}(\mathcal H)$ with closed ranges such that  $ST$ has closed range.  Then  $(ST)^\dagger=T^\dagger S^\dagger$ if and only if $\mathcal{R}(S^*ST)\subseteq \mathcal{R}(T)$ and $\mathcal{R}(TT^*S^*)\subseteq \mathcal{R}(S^*)$.
\end{pro}

\begin{cor}\label{ROL-autoadjuntos}
Let $S,T\in\mathcal{L}^h$ be closed range operators such that   $ST=TS$ has closed range.  Then  $(ST)^\dagger=T^\dagger S^\dagger$.
\end{cor}

\begin{proof}
It is straightforward from Proposition \ref{dagger de AB con rg iguales}. 
\end{proof}

The following result characterizes the antitonicity property for the Moore-Penrose operator on the set $\mathcal{L}^+$.  Its proof  can be found in \cite{fongi2023moore}.

\begin{teo}\label{lowner-order}
Consider $S,T\in \mathcal{L}^+$  closed range operators. Then any two of the following conditions imply the third condition:
\begin{enumerate}
\item $S\leq T$;
\item $T^\dagger\leq S^\dagger$;
\item $\mathcal{R}(S)\cap\mathcal{N}(T)=\mathcal{R}(T)\cap\mathcal{N}(S)=\{0\}$.
\end{enumerate}
\end{teo}

As a consequence of Douglas' theorem we can characterize the L\"owner order between positive operators by means of the spectral radius of certain product of operators and a range inclusion condition. The next result is the extension to the infinite dimensional case of  a result due to Baksalary, Liski and Trenkler \cite{MR1032688} and \cite{AG-splitting}. 
\begin{pro}\label{BLT}
	Consider $S,T\in \mathcal L^+$ such that $\mathcal R(T)$ is closed. Then, $S\leq T$ if and only if $\rho(S^{1/2}{T^\dagger S^{1/2}})\leq 1$ and $\mathcal{R}(S^{1/2})\subseteq \mathcal{R}(T)$.
\end{pro}

\begin{proof}
	Let $S,T\in\mathcal L^+$. If $S\leq T$ then, by Douglas' theorem, $\mathcal{R}(S^{1/2})\subseteq \mathcal{R}(T^{1/2})=\mathcal R(T)$ and  $\|(T^{1/2})^\dagger S^{1/2}\|=\inf\{\lambda: S\leq \lambda T\}$.  Therefore, $\|(T^{1/2})^\dagger S^{1/2}\|\leq 1$ and $\rho(S^{1/2}T^\dagger S^{1/2})=\| S^{1/2}T^\dagger S^{1/2}\|=\|(T^{1/2})^\dagger S^{1/2}\|^2\leq 1$.  
	
	Conversely, suppose that  $\rho(S^{1/2}T^\dagger S^{1/2})\leq 1$ and $\mathcal{R}(S^{1/2})\subseteq \mathcal{R}(T)=\mathcal{R}(T^{1/2})$. Then $\|(T^{1/2})^\dagger S^{1/2}\|^2 = \|S^{1/2}T^\dagger S^{1/2}\|=\rho(S^{1/2}T^\dagger S^{1/2})\leq 1$. Hence, applying again Douglas' theorem we get that,  $\inf\{\lambda: S\leq \lambda T\}=\|(T^{1/2})^\dagger S^{1/2}\|\leq 1$ and $S\leq T$.
	\end{proof}

On the other hand, the group inverse  of an operator is defined for the class of split operators; $T\in \mathcal L(\mathcal{H})$ is a {\it{split operator}} if $\mathcal{R}(T)\dotplus \mathcal{N}(T)=\mathcal{H}$. If $T\in\mathcal{L}(\mathcal{H})$ is a split operator then there exists a unique  $T^\sharp \in \mathcal L(\mathcal{H})$ such that the following equations are simultaneously satisfied:
$$
TXT=T, \ \ \ \ XTX=X, \ \ \ \ \ TX=XT.
$$
The operator $T^\sharp$ is called the {\it{group inverse}} of $T$. In \cite{Robert} necessary and sufficient conditions for the existence of this kind of pseudoinverse are shown:
\begin{teo}\label{group inverse conditions}
	Consider  $T\in \mathcal{L}(\mathcal{H})$, then the following statements are equivalent:
	\begin{enumerate}
		\item  $\mathcal H=\mathcal{R}(T)\dotplus \mathcal{N}(T)$;
		\item $\mathcal{R}(T^2)=\mathcal{R}(T)$ and  $\mathcal{N}(T^2)=\mathcal{N}(T)$;
		\item $T^\sharp $ exists.
	\end{enumerate}
\end{teo}
It is well known that if $T$ admits a group inverse then it has  closed range, $\mathcal{R}(T^\sharp)=\mathcal{R}(T)$, $\mathcal{N}(T^\sharp)=\mathcal{N}(T)$ and $TT^\sharp=T^\sharp T=Q_{\mathcal{R}(T)//\mathcal{N}(T)}$. In addition, $T$ admits a group inverse if and only if $T^*$ admits a group inverse. On the other hand, if $T \in \mathcal{L}(\mathcal{\mathcal{H}})$  is an EP operator (i.e., $T$ has closed range and  $\mathcal{R}(T)=\mathcal{R}(T^*)$)  then it admits a group inverse. Moreover, $T^\sharp=T^\dagger$ if and only if $T$ is an EP operator.

\

Remember that every $T\in\mathcal{L}(\mathcal{H})$ admits a factorization $T=U|T|$, where $U$ is a partial isometry and $|T|=(T^*T)^{1/2}$. Moreover, this factorization is unique if the condition $\mathcal{N}(U)=\mathcal{N}(T)$ is imposed.  The factorization $T=U_T|T|$, where $U_T$ is the partial isometry such that  $\mathcal{N}(U_T)=\mathcal{N}(T)$ is called the polar decomposition of $T$. It is not dificult to see that $T^\dagger=|T|^\dagger U_T^*$. In addition, if  $T\in\mathcal{L}^h$ then $T\leq |T|$.

\
 
 We finish this section by introducing some order relations on $\mathcal{L}(\mathcal{H})$ that will be used throughout the article.
 
\begin{defi}\label{ordenes}
	Let $S,T\in \mathcal{L}(\mathcal{H})$. Then:
	\begin{enumerate}
		\item $S\overset{*}\leq T$ if   there exist $P,Q\in \mathcal{P}$ such that $S=PT$ and $S^*=QT^*$. The orthogonal projections can be chosen such that $P=P_S$ and $Q=P_{S^*}$.
		
		\item $S\overset{-}\leq T$ if  there exist $P,Q\in \mathcal{Q}$ such that $S=PT$ and $S^*=QT^*$. The ranges of $P$ and $Q$ can be fixed as $\mathcal R(P)=\overline{\mathcal R(S)}$ and $\mathcal R(Q)=\overline{\mathcal R(S^*)}$.
		
		\item  $S\overset{\sharp}\leq T$ if $S=T$ or if  there exists $Q\in\mathcal{Q}$ with $\mathcal{R}(Q)=\mathcal{R}(S)$, $\mathcal{N}(Q)=\mathcal{N}(S)$ such that $S=QT=TQ$.
	\end{enumerate}
\end{defi}

The above order relations are known as {\it{star order, minus order}} and {\it{sharp order}}, respectively. The reader can be referred to \cite{AntCanSto2010}, \cite{Dji2017Theminusordes}, \cite{Efimov}, \cite{jose2015partial} and references therein for different  treatments on these order relations. The following result gives  characterizations of these orders by means of  operator range decompositions. The proof can be  found in  \cite{Dji2017Theminusordes} and \cite{fongi2023moore}. 

\begin{pro}\label{caracterizaciones ordenes}
	Consider $S,T\in \mathcal{L}{(\mathcal{H})}$. Then,
	\begin{enumerate}
		\item $S\overset{-}\leq T$ if  $\mathcal{R}(T)=\mathcal{R}(S) \overset{.}{+}\mathcal{R}(T-S)$ and $\mathcal{R}(T^*)=\mathcal{R}(S^*) \overset{.}{+}\mathcal{R}(T^*-S^*)$.
		\item $S\overset{*}\leq T$ if and only if  $\mathcal{R}(T)=\mathcal{R}(S) \oplus\mathcal{R}(T-S)$ and $\mathcal{R}(T^*)=\mathcal{R}(S^*) \oplus\mathcal{R}(T^*-S^*)$.
		\item If  $S,T$ are group invertible then $S\overset{\sharp}\leq T$  if and only if $\mathcal{R}(T)=\mathcal{R}(S) \dot +\mathcal{R}(T-S)$ where  $\mathcal{R}(T-S)\subseteq \mathcal{N}(S)$ and $\mathcal{R}(S)\subseteq \mathcal {N}(T-S)$.
	\end{enumerate}
\end{pro}

\section{Proper splittings of operators and an iterative process}

In \cite{MR348984} the concept of proper splitting of a matrix $T$ is applied to approximate the minimum least square solution of a linear system $Tx=w.$ In \cite{AG-splitting} this kind of splitting is used  to guarantee the convergence of an iterative process to  the reduced solution for $\mathcal{M}$ of a solvable matrix equation $TX=W$, where $T,W\in\mathbb{C}^{m\times n}$ and $\mathcal{M}$ is a  subspace of $\mathbb{C}^n$ such that $\mathcal{M}\dot+\mathcal{N}(T)=\mathbb{C}^n$. There, sufficient conditions for the convergence of the iterative process were given. Our goal is to extend the study made in  \cite{AG-splitting} to the infinite dimensional context. 

\begin{defi}
Consider $T\in\mathcal{L}(\mathcal{H})$ with closed range. A proper splitting of $T$  is a decomposition  $T=U-V$, where $U,V\in\mathcal{L}(\mathcal{H})$ and $\mathcal{R}(U)=\mathcal{R}(T)$ and $\mathcal{N}(U)=\mathcal{N}(T)$.
\end{defi}

Consider $T\in\mathcal{L}(\mathcal{H})$ with closed range, $W\in\mathcal{L}(\mathcal{H})$ such that  $\mathcal{R}(W)\subseteq \mathcal{R}(T)$ and a closed subspace $\eme\subseteq \mathcal{H}$ such that $\eme\overset{.}{+}\mathcal{N}(T)=\mathcal{H}$. Let $T=U-V$ be a proper splitting of $T$. Consider $Y_\eme, Z_\eme \in \mathcal{L}(\mathcal{H})$  the reduced solutions for $\eme$ of $UY=V$ and $UZ=W,$ respectively. We define, as in \cite{AG-splitting},  {\it{the iterative process for $\mathcal{M}$ of the proper splitting $T=U-V$ with respect to $W$}}:

\begin{equation}
X^{i+1}=Y_\eme X^i+ Z_\eme  .\label{proceso iterativo}
\end{equation}

As in the finite dimensional context \cite{AG-splitting}, in the infinite dimensional case it can be proved  that if the iteration (\ref{proceso iterativo}) converges, then it converges  to the reduced solution for $\eme$ of $TX=W.$  We omit the proof of the following two results because they are similar to those given in \cite{AG-splitting}.

\begin{teo}\label{convergencia}
Consider $T\in \mathcal{L}(\mathcal{H})$ with closed range,  $W\in \mathcal{L}(\mathcal{H})$ such that $\mathcal{R}(W)\subseteq \mathcal{R}(T)$ and $\eme$ a closed subspace of $\mathcal{H}$ such that $\eme\overset{.}{+}\mathcal{N}(T)=\mathcal{H}$.  Consider the proper splitting $T= U-V$ of $T$. Then the iterative process (\ref{proceso iterativo}) converges for all $X^0$ to $X_\eme$, the reduced solution for $\eme$ of the equation  $TX=W,$ if and only if $\rho(Y_\eme)<1$. 
\end{teo}

\begin{cor}\label{corrho}
Let $T\in \mathcal{L}(\mathcal{H})$ be a  closed range operator  and   $T= U-V$  a proper splitting of $T$.  $W\in \mathcal{L}(\mathcal{H})$ such that $\mathcal{R}(W)\subseteq \mathcal{R}(T)$. Then, the iterative process of the proper splitting $T=U-V$ converges for some closed subspace $\eme$ such that $\eme\overset{.}{+}\mathcal{N}(T)=\mathcal{H}$ if and only if  the iterative process of the proper splitting $T=U-V$ converges for all closed subspace $\eme$ such that $\eme\overset{.}{+}\mathcal{N}(T)=\mathcal{H}.$ In particular, the iterative process (\ref{proceso iterativo}) is convergent if and only if $\rho(U^\dagger V)<1$.
\end{cor}

Given $T\in\mathcal{L}(\mathcal{H})$ with closed range, the aim of this section is to obtain sufficient conditions that guarantee the convergence of a general proper splitting $T=U-V$. We first collect some properties of  proper splittings that will be useful.
\begin{pro}\label{propiedades generales splitting}
Consider $T\in\mathcal{L}(\mathcal{H})$ a closed range operator. If $T=U-V$ is a proper splitting of $T$ then the following assertions follow:  
\begin{enumerate}
\item $(U^\dagger T)^\dagger =T^\dagger U$;
\item $\mathcal{N}(U^\dagger V)=\mathcal{N}(V)$
\item $T^\dagger=(I-U^\dagger V)^{-1}U^\dagger$.
\end{enumerate}
\end{pro}

\begin{proof}
We only prove item 3 because the proofs of items 1 and 2  are similar to those in the finite dimensional case. 

3. Let $T=U-V$ be a proper splitting of $T$. First, let us see that $I-U^\dagger V$ is an invertible operator. In fact, observe that $U^\dagger V=U^\dagger(U-T)=P_{T^*}-U^\dagger T$, then
 $I-U^\dagger V=P_{\mathcal{N}(T)}+U^\dagger T$. In order to see that this operator is invertible it is sufficient to note that $(P_{\mathcal{N}(T)}+(U^\dagger T)^\dagger)(P_{\mathcal{N}(T)}+U^\dagger T)=(P_{\mathcal{N}(T)}+U^\dagger T)(P_{\mathcal{N}(T)}+(U^\dagger T)^\dagger)=I$. Therefore, $(I-U^\dagger V)^{-1}=P_{\mathcal{N}(T)}+(U^\dagger T)^\dagger$. Now, let us see that $T^\dagger=(I-U^\dagger V)^{-1}U^\dagger $. In fact, since $T(I-U^\dagger V)^{-1}U^\dagger =T(P_{\mathcal{N}(T)}+(U^\dagger T)^\dagger)U^\dagger =T(U^\dagger T)^\dagger U^\dagger =TT^\dagger UU^\dagger =P_T$ and $\mathcal{R}((I-U^\dagger V)^{-1}U^\dagger)=\mathcal{R}((P_{\mathcal{N}(T)}+(U^\dagger T)^\dagger)U^\dagger )=\mathcal{R}((U^\dagger T)^\dagger U^\dagger )=\mathcal{R}(T^\dagger)=\mathcal{N}(T)^\bot$ then the assertion follows. 
\end{proof}

\begin{pro}\label{polarpositivo}
Let $T\in\mathcal{L}(\mathcal{H})$ be a closed range operator. If $T=U-V$ is a proper splitting of $T$ then the following assertions are equivalent:
\begin{enumerate}
\item $T^\dagger V\in \mathcal{L}^+$;
\item $U^\dagger V \in  \mathcal{L}^+$;
\item $0\leq U^\dagger V\leq P_{V^*}$.
\item $0\leq U^\dagger T\leq P_{T^*}$;
\item The equation $UX=V$ has a postive solution;
\item $U^\dagger T\in\mathcal{L}^h$ and $(P_{T^*}-U^\dagger T)^2\leq \lambda (P_{T^*}-U^\dagger T)$,  for some $\lambda\geq 0$.
\end{enumerate}
\end{pro}
\begin{proof}
$1\leftrightarrow 2.$ Suppose that $T^\dagger V\in\mathcal{L}^+$. Since $T^\dagger V=T^\dagger (U-T)= T^\dagger U-P_{T^*}$ then we get that $P_{T^*}\leq T^\dagger U$. Now, as $\mathcal{R}(P_{T^*})\cap\mathcal{N}(T^\dagger U)=\{0\}$ and $\mathcal{R}(T^\dagger U)\cap\mathcal{N}(P_{T^*})=\{0\}$ then, by Theorem \ref{lowner-order}, $U^\dagger T=(T^\dagger U)^\dagger \leq P_{T^*}$. Therefore, $U^\dagger V=U^\dagger (U-T)=P_{T^*}-U^\dagger T\in \mathcal{L}^+$.  The converse is similar.

$1\leftrightarrow 3.$ Let $T=U-V$ be a proper splitting of $T$. Since $T^\dagger =(I-U^\dagger  V)^{-1}U^\dagger $ then
\begin{eqnarray}
0\leq T^\dagger V &\leftrightarrow& 0\leq (I-U^\dagger V)^{-1}U^\dagger V \leftrightarrow 0 \leq \left((I-U^\dagger  V)^{-1}U^\dagger  V \right)^\dagger  \nonumber \\
&\leftrightarrow& 0\leq \left(U^\dagger V \right)^\dagger (I-U^\dagger  V) \leftrightarrow  0\leq \left(U^\dagger  V \right)^\dagger - P_{\mathcal{N}(U^\dagger V)^\bot}\nonumber\\
&\leftrightarrow& 0\leq P_{V^*}\leq \left(U^\dagger  V\right)^\dagger, \label{TdaggerV}
\end{eqnarray}
 where the third equivalence follows by Corollary \ref{ROL-autoadjuntos} and by equivalence $1.\leftrightarrow 2.$ of this proposition. Now, if $0\leq P_{V^*}\leq \left(U^\dagger V\right)^\dagger$ then by Theorem \ref{lowner-order}, we get $0\leq U^\dagger V \leq P_{V^*}$. Conversely, if $0\leq U^\dagger V \leq P_{V^*}$ then Theorem \ref{lowner-order}, $0\leq P_{V^*}\leq \left(U^\dagger V\right)^\dagger$. Then, by the equivalences (\ref{TdaggerV}), the assertion follows.

$2\leftrightarrow 4$. Observe  that  $0\leq U^\dagger T\leq P_{T^*}$ if and only if  $U^\dagger V=P_{T^*}-U^\dagger T\in\mathcal{L}^+$.

$2\leftrightarrow 5$. If $U^\dagger V\in\mathcal{L}^+$ then, by Douglas theorem, the equation $UX=V$ has a positive solution. Conversely, if $UX=V$ has a positive solution then $VV^*\leq \lambda UV^*$ for some $\lambda \geq 0$, see \cite{Seb} or \cite{AG-PA}. Now, multiplying on the left and right by $U^\dagger$ and $(U^\dagger)^*$, respectively,  the assertion follows. 

$5\leftrightarrow 6$. If $UX=V$ has a positive solution then $VV^*\leq \lambda UV^*$ for some $\lambda \geq 0$, see \cite{Seb} or \cite{AG-PA}.  Now, multiplying on the left and right by $U^\dagger$ and $(U^\dagger)^*$, respectively,  and since $U^\dagger V=P_{T^*}-U^\dagger T$, the assertion follows. The converse is similar.
\end{proof}

\begin{pro}\label{polarcompacto}  Consider $T\in\mathcal{L}(\mathcal{H})$ with closed range. If $T=U-V$ is a proper splitting of $T$ then  the following statements are equivalent:
\begin{enumerate}
\item $T^\dagger V\in\mathcal{K}$;
\item $U^\dagger V\in\mathcal{K}$;
\item $V\in\mathcal{K}$;
\end{enumerate}
\end{pro}

\begin{proof}
$1 \leftrightarrow 2$. If $T^\dagger V=T^\dagger (U-T)=T^\dagger U- P_{T^*}\in \mathcal{K}$ then $(T^\dagger U)^\dagger(T^\dagger U-P_{T^*})\in\mathcal{K}$. Now observe that $ (T^\dagger U)^\dagger(T^\dagger U-P_{T^*})= P_{T^*}-U^\dagger T=U^\dagger V$. The converse is similar. 

$1\leftrightarrow 3$ Suposse $T^\dagger V\in\mathcal{K}$. Since $\mathcal{R}(V)\subseteq \mathcal{R}(T)$ then $TT^\dagger V=V\in\mathcal{K}$. Conversely, if $V\in\mathcal{K}$ it is clear that $T^\dagger V\in\mathcal{K}$.
\end{proof}

We can now  prove a sufficient condition to guarantee the convergence of a general proper splitting of a closed range operator.  The proof is similar to the one given for the finite dimensional case \cite{AG-splitting}, however we include it for the sake of completeness.

\begin{teo}\label{condiciones suficientes de convergencia}
Consider $T\in\mathcal{L}(\mathcal{H})$ with closed range and $T=U-V$ a proper splitting of $T$. If $U^\dagger V\in\mathcal{L}^+\cap\mathcal{K}$ then the proper splitting of $T$ converges. Moreover, it holds that $\rho(U^\dagger V)=\displaystyle\frac{\rho(T^\dagger V)}{1+\rho(T^\dagger V)}<1$.
\end{teo}

\begin{proof}
 Note that $U^\dagger V\in\mathcal{L}^+\cap\mathcal{K}$ if and only if  $T^\dagger V\in\mathcal{L}^+\cap\mathcal{K}$. Now, let $0\neq\lambda$. Let us see that $\lambda\in\sigma(T^\dagger V)$ if and only if $\frac{\lambda}{1+\lambda}\in\sigma(U^\dagger V)$. In fact, 
\begin{eqnarray}
T^\dagger Vx=\lambda x &\leftrightarrow & (I-U^\dagger V)^{-1}U^\dagger  Vx=\lambda x \leftrightarrow U^\dagger  Vx= (I-U^\dagger  V) \lambda x\nonumber\\
&\leftrightarrow& U^\dagger  V(1+\lambda)x=\lambda x \leftrightarrow U^\dagger  Vx=\displaystyle\frac{\lambda}{1+\lambda} x. \nonumber 
\end{eqnarray}
Since that $T^\dagger V\in\mathcal{L}^+\cap\mathcal{K}$, then $\rho(T^\dagger V)\in \sigma(T^\dagger V)$ and $\displaystyle\frac{\lambda}{1+\lambda}$ achieves its maximum when $\lambda=\rho(T^\dagger V)$. Therefore $\rho(U^\dagger V)=\displaystyle\frac{\rho(T^\dagger V)}{1+\rho(T^\dagger V)}<1$.   
\end{proof}

\section{The polar proper  splitting}

The polar proper splitting for  rectangular matrices was defined in \cite{AG-splitting}. Next we extend this  particular proper splitting for  Hilbert space operators. 

\begin{defi} 
Consider $T\in \mathcal{L}(\mathcal{H})$ a closed range operator. The polar proper splitting of $T$ is $T=U_T-V$, where $U_T$ is the partial isometry of the polar decomposition of $T$.
\end{defi}

By Theorem \ref{condiciones suficientes de convergencia},  given $T\in\mathcal{L}(\mathcal{H})$ such that   $T=U-V$ is a proper splitting of $T$ and $U^\dagger V \in\mathcal{K}$, the positivity of $U^\dagger V$ is a sufficient condition to guarantee the convergence of the proper splitting of $T$. In the particular case that  $T=U_T-V$ is the polar proper splitting of $T$, note that $U_T^\dagger=U_T^*$. In the next result  we see that condition  $U_T^*V\in\mathcal{L}^+$ can be characterized by means of the norm of  $T$.

\begin{pro}
Consider $T\in\mathcal{L}(\mathcal{H})$ with closed range. If $T=U_T-V$ is the polar proper splitting of $T$ then the following assertions are equivalent:
\begin{enumerate}
\item $U_T^*V\in\mathcal{L}^+$;
\item $\|T\|\leq 1$;
\item $(P_{T}-|T^*|)^2\leq \lambda (|T^*|-P_T)$ for some $\lambda \geq 0$.
\end{enumerate}
\end{pro}

\begin{proof}
	$1.\leftrightarrow 2.$ Recall that $U_T^*V=P_{T^*}-|T|$. If $U_T^*V\geq 0$ then $|T|\leq P _{T^*}$, so that $\|T\|=\| |T|\|\leq 1$. Conversely, if $\|T\|\leq 1$ then $\rho(|T|^{1/2}P_{T^*}|T|^{1/2})=\||T|^{1/2}P_{T^*}|T|^{1/2}\| =\||T|\|=\|T\|\leq 1$. Moreover, $\mathcal{R}(|T|^{1/2})=\mathcal{R}(P_{T^*})$. Then,  by Proposition \ref{BLT}, it follows  that $|T|\leq P_{T^*}$. Therefore, $U_T^*V\in  \mathcal{L}^+$.

$1.\leftrightarrow 3.$ If  $U_T^*V\in\mathcal{L}^+$ then the equation $U_TX=V$ has a positive solution. Then by \cite{Seb} or \cite[Theorem 2.3]{AG-PA}, we get that $VV^*\leq \lambda U_TV^*$ for some $\lambda \geq 0$. So that, $(P_T-|T^*|)^2\leq \lambda (|T^*|-P_T)$, for some $\lambda \geq 0$.  Conversely, since $VV^*=(P_T-|T^*|)^2\leq \lambda (|T^*|-P_T)=\lambda U_TV^*$, for some $\lambda \geq 0$ then $U_TV^*\geq 0$. So that $U_T^*U_TV^*U_T=V^*U_T\geq 0$. Then $U_T^*V\in\mathcal{L}^+$. 
\end{proof}

Now, we get a characterization of the convergence of the polar proper splitting. 

\begin{teo}\label{convergencia del polar}
Let  $T\in \mathcal{L}(\mathcal{H})$ be a closed range operator and $T=U_T-V$ the polar proper splitting of $T$. Then the following assertions are equivalent:
\begin{enumerate}
\item  The polar proper splitting of $T$  converges;
\item  $\|U_T-T\|<1$;
\item $\|P_{T^*}-|T|\|<1$;
\item $\|T\|<2$.
\end{enumerate}
\end{teo}

\begin{proof}
Since $U_T^*V=P_{T^*}-|T|\in   \mathcal{L}^h$ then $\rho(U_T^*V)=\|U_T^*V\|$. Now, the equivalences $1.\leftrightarrow 2. \leftrightarrow 3.$ follow from the fact that $\|U_T^*V\|=\|P_{T^*}-|T|\|=\|U_T^*(U_T-T)\|=\|U_T-T\|$. 
To prove equivalence $2. \leftrightarrow 4.$  observe that by \cite[Lemma 2.1]{Chium2021Onaconjeture} it holds that $\|U_T-T\|=max\{1-\gamma(T), \|T\|-1\}$, where $\gamma(T)$ denotes the reduced minimum modulus of $T$. Then the assertion follows noticing that $\gamma(T)>0$ because $T$ has closed range.
\end{proof}

In the next result we show that given $S,T\in\mathcal{L}(\mathcal{H})$ with closed range such that $S\overset{*}\leq T$ then the convergence of the polar proper splitting of $T$ guarantees the convergence of the polar proper splitting of $S$. In addition,  in this case, the polar proper splitting of $T$ induces the polar proper splitting of $S$.

\begin{teo}\label{polar sp orden estrella}
Consider $S,T\in\mathcal{L}(\mathcal H)$ with closed range such that $S\overset{*}{\leq}T$. If  the polar proper splitting of $T$ converges then the polar proper splitting of $S$ converges. Moreover, $\rho(U_S^*W)\leq \rho(U_T^*V)$, where $S=U_S-W$ and $T=U_T-V$ are the polar proper splitting of $S$ and $T$, respectively. In addition, the polar proper splitting of $S$ can be obtained from the polar proper splitting of $T$ as follows: $S=P_SU_T-P_SV$.
\end{teo}

\begin{proof}
If $S\overset{*}{\leq}T$ then $S=P_ST$. Now, if the polar proper splitting of $T$ converges then, by Theorem \ref{convergencia del polar}, we get that $\|T\|<2$. Then, $\|S\|=\|P_ST\|\leq \|P_S\|\|T\|=\|T\|<2$ and so that, the polar proper splitting of $S$ converges.  Moreover, since  $S\overset{*}{\leq}T$ then by \cite[Theorem 2.15]{AntCanSto2010}, it holds that  $|S|\overset{*}{\leq}|T|$. Therefore,  $|S|=P_{S^*}|T|$. Then $\rho(U_S^*W)=\|P_{S^*}-|S|\|=\|P_{S^*}-P_{S^*}|T|\|=\|P_{S^*}(I-|T|)\|=\|P_{S^*}P_{T^*}(I-|T|)\|\leq\|P_{T^*}(I-|T|)\|=\rho(U_T^*V)<1$. Finally, since $S\overset{*}\leq T$ then, again by \cite[Theorem 2.15]{AntCanSto2010}, it holds that $U_S\overset{*}\leq U_T$ i.e., $U_S=P_SU_T=U_TP_{S^*}$.  
Now,  $S=U_S-W=P_SU_T-W$ then $W=P_SU_T-P_ST=P_S(U_T-T)=P_SV$. Hence the assertion follows.
\end{proof}

\begin{rem}
In Theorem \ref{polar sp orden estrella} the convergence of the polar proper splitting of $S$ can be strictly faster than the convergence of the polar proper splitting of $T$. In fact, consider $\mathcal{H}=
\mathbb{C}^3$,  $S=\left(\begin{matrix}
 0& 0 & 0\\
0 & 1/2 & 0\\
0 & 1/2 & 0
\end{matrix}\right)$ and $T=\left(\begin{matrix}
 1/2 & 0 & 0\\
0 & 1/2 & 0\\
0 & 1/2 & 0
\end{matrix}\right)$. Since $\mathcal{R}(T)=\mathcal{R}(S)\oplus\mathcal{R}(T-S)$ and $\mathcal{R}(T^*)=\mathcal{R}(S^*)\oplus\mathcal{R}(T^*-S^*)$ then $S\overset{*}\leq T$. Now, the polar proper splitting of $S$ is $S=U_S-W=\left(\begin{matrix}
 0& 0 & 0\\
0 & \sqrt{2}/2 & 0\\
0 & \sqrt{2}/2 & 0
\end{matrix}\right) - \left(\begin{matrix}
 0& 0 & 0\\
0 & \sqrt{2}/2 -1/2 & 0\\
0 & \sqrt{2}/2 - 1/2 & 0
\end{matrix}\right)$ and the polar splitting of $T$ is $T=U_T-V=\left(\begin{matrix}
 1& 0 & 0\\
0 & \sqrt{2}/2 & 0\\
0 & \sqrt{2}/2 & 0
\end{matrix}\right) - \left(\begin{matrix}
 1/2 & 0 & 0\\
0 & \sqrt{2}/2 -1/2 & 0\\
0 & \sqrt{2}/2 -1/2 & 0
\end{matrix}\right)$. Then it holds that $\rho(U_S^*W)=\frac{2-\sqrt{2}}{2}\lneq\rho(U_T^*V)=\frac{1}{2}$.

\end{rem}

\begin{rem}
 Condition  $S\overset{*}{\leq}T$ in Theorem \ref{polar sp orden estrella} can not be replaced by a weaker order condition like  $S\overset{-}{\leq}T$. In fact, consider $\mathcal{H}=\mathbb{C}^3$. Take 
$S=\left(\begin{matrix}
1 & 0 & 0\\
2 & 0 & 0\\
0 & 0 & 0
\end{matrix}\right)$ and $T=\left(\begin{matrix}
1 & 0 & 0\\
0 & 1/2 & 0\\
0 & 0 & 0
\end{matrix}\right)$. It is clear that $\mathcal{R}(T)=\mathcal{R}(S)\dot+\mathcal{R}(T-S)$ and $\mathcal{R}(T)=\mathcal{R}(S^*)\dot+\mathcal{R}(T-S^*)$ (remember that $T=T^*$).  So that, $S\overset{-}{\leq}T$ but they are not related with the star order. Now, since $P_T-|T|=\left(\begin{matrix}
1 & 0 & 0\\
0 & 1& 0\\
0 & 0 & 0
\end{matrix}\right) - 
\left(\begin{matrix}
1 & 0 & 0\\
0 & 1/2 & 0\\
0 & 0 & 0
\end{matrix}\right)=
\left(\begin{matrix}
0& 0 & 0\\
0 & 1/2 & 0\\
0 & 0 & 0
\end{matrix}\right)$ then $\|P_T-|T|\|<1$ and thus the polar proper splitting of $T$ converges. On the other hand, $P_{S^*}-|S|=
\left(\begin{matrix}
1 & 0 & 0\\
0 & 0& 0\\
0 & 0 & 0
\end{matrix}\right) - 
\left(\begin{matrix}
\sqrt{5}& 0 & 0\\
0 & 0 & 0\\
0 & 0 & 0
\end{matrix}\right)=
\left(\begin{matrix}
1-\sqrt{5}& 0 & 0\\
0 & 0 & 0\\
0 & 0 & 0
\end{matrix}\right)$. Then $\|P_{S^*}-|S|\|=|1-\sqrt{5}|>1$. Then, the polar proper splitting of $S$ does not converge.
\end{rem}

\begin{rem}
Condition  $S\overset{*}{\leq}T$ in Theorem \ref{polar sp orden estrella} can not be replaced by condition  $S\overset{\sharp}{\leq}T$. In fact, consider  $T= \left(\begin{matrix}
3/2& 0 & 0\\
0 & 0 &  0\\
0 & 0 &3/2
\end{matrix}\right)$ and 
$S=\left(\begin{matrix}
3/2& 0 & 9/2\\
0 & 0 &  0\\
0 & 0 &0
\end{matrix}\right)$.
Then  $S\overset{\sharp}{\leq}T$ because $S=QT=TQ$, where 
 $Q= \left(\begin{matrix}
1&0 & 3\\
0 & 0 &0\\
0 & 0 &0
\end{matrix}\right)$. Also, note that  the polar proper splitting of $T$ converges because $\|T\|=\frac{3}{2}<2$. However, the polar proper splitting of $S$ does not converge because $\|S\|>2$.
\end{rem}

\begin{pro}
Let $S\in\mathcal{P}\cdot\mathcal{L}^h$. If $S=P_ST$, where $T\in\mathcal{L}^h$ and the polar proper splitting of $T$ converges then the polar proper splitting of $S$ converges. Conversely, if the polar proper splitting of $S$ converges then there exists $T_0\in\mathcal{L}^h$ such that $S=P_ST_0$ and the polar proper splitting of $T_0$ converges.
\end{pro}

\begin{proof}
Let $S=P_ST$, where $T\in \mathcal{L}^h$. Consider the set $\mathcal{A}_S=\{B\in\mathcal{L}^h: S=P_SB\}$. By \cite[Theorem 3.2]{AG-PB},  $\|S\|=\underset{B\in\mathcal{A}_S}{\min}\|B\|$. Therefore, if the polar proper splitting of $T$ then, by Theorem \ref{convergencia del polar}, $\|S\|\leq\|T\|<2$ and  so, the polar proper splitting of $S$ converges. Conversely, if the polar proper splitting of $S$ converges then $\|S\|<2$. Then applying again \cite[Theorem 3.2]{AG-PB}, there exists $T_0\in\mathcal{A}_S$ such that $\|T_0\|=\|S\|$. Therefore, the polar proper splitting of $T_0$ converges.
\end{proof}

\section{Splittings for split operators}

In this section we put the focus on splittings of split operators. Recall that $T\in \mathcal{L}(\mathcal H)$ is a split operator if $T$ has closed range and $\mathcal{R}(T)\dot + \mathcal{N}(T)=\mathcal{H}$.  For this class of operators we consider the following proper splittings:

\begin{defi}
Given a split operator  $T\in \mathcal{L}(\mathcal H)$, we say that:
\begin{enumerate}
\item $T=Q_T-V$ is the $Q_T$-proper splitting of $T$, where $Q_T=Q_{\mathcal{R}(T)//\mathcal{N}(T)}$.
\item $T=T^\sharp-V$ is the group proper splitting of $T$, where $T^\sharp$ is the group inverse of $T$.
\end{enumerate}
\end{defi}

\begin{pro}\label{Q-proper}
Let  $T\in \mathcal{L}(\mathcal{H})$ be a split operator such that $\|P_{T^*}(I-T)\|<1$. Then the $Q_T$-proper splitting of $T$ converges. In addition, if $P_{T^*}T\in\mathcal{L}^h$ and the $Q_T$-proper splitting of $T$ converges then $\|P_{T^*}(I-T)\|<1$.
\end{pro}

\begin{proof}
Let $T=Q_T-V$ be the $Q_T$ proper splitting of $T$. By  \cite[Theorem 4.1]{MR2653816}, it holds that $Q_T^\dagger V=P_{T^*}P_TV=P_{T^*}V=P_{T^*}(Q_T-T)=P_{T^*}-P_{T^*}T$. Now, since $\rho(Q_T^\dagger V)\leq \|P_{T^*}(I-T)\|<1$  then $T=Q_T-V$ converges. The converse follows from the above argument plus the fact that since  $Q_T^\dagger V\in\mathcal{L}^h$ then it holds $\rho(Q_T^\dagger V)=\|Q_T^\dagger V\|$.
\end{proof}

Next, we study conditions for the  convergence of the $Q_T$-proper splitting for different classes of operators. Namely, we will consider the sets
$\mathcal{P}\cdot\mathcal{P}=\{T\in\mathcal{L}(\mathcal{H}): T=PQ,  \ \textrm{for}  \ P,Q\in\mathcal{P}\}$, $\mathcal{P}\cdot\mathcal{L}^+=\{T\in\mathcal{L}(\mathcal{H}): \ T=PA, \ \textrm{for} \ P\in\mathcal{P} \  \textrm{and} \ A\in\mathcal{L}^+\}$ and $\mathcal{P}\cdot\mathcal{L}^h=\{T\in\mathcal{L}(\mathcal{H}): \ T=PB, \ \textrm{for} \ P\in\mathcal{P} \  \textrm{and} \ B\in\mathcal{L}^h\}$.  The sets $\mathcal{P}\cdot\mathcal{P}$, $\mathcal{P}\cdot\mathcal{L}^+$ and  $\mathcal{P}\cdot\mathcal{L}^h$ have been studied in \cite{CM}, \cite{AG-PA} and  \cite{AG-PB}, respectively. From the definitions is evident that the following inclusions hold: $\mathcal{P}\cdot\mathcal{P} \subseteq\mathcal{P}\cdot\mathcal{L}^+\subseteq \mathcal{P}\cdot\mathcal{L}^h$. Moreover, these inclusions are strict in general \cite{AG-PA, AG-PB}. 
They are not subsets of  the class of split operators, in general. However, under certain extra conditions, their elements are split operators. In fact, it holds that: $T\in\mathcal{P}\cdot\mathcal{P}$ is a split operator if and only if $\mathcal{R}(T)$ is closed   \cite[Theorem 3.2]{CM};  if $\mathcal{R}(T)$ is closed  then $T\in\mathcal{P}\cdot\mathcal{L}^+$ if and only if $T$ is a split operator and $TP_T\in\mathcal{L}^+$ \cite[Theorem 3.3]{AG-PA};  if $\mathcal{R}(T)$ is closed and $T\in\mathcal{P}\cdot \mathcal{L}^h$ then, there  exists $B\in\mathcal{L}^h$ with $\mathcal{N}(B)=\mathcal{N}(T)$ such that $T=P_TB$ if and only if $T$ is a split operator \cite[Corollary 2.13]{AG-PB}.

\begin{cor} The following assertions hold:
\begin{enumerate}
\item Let  $T\in \mathcal{L}(\mathcal{H})$ be  a split operator such that $T^*\in\mathcal{P}\cdot\mathcal{L}^h$. Then, the $Q_T$-proper splitting of $T$ converges if and only if $\|P_{T^*}(I-T)\|<1$.
\item Let $T\in \mathcal{L}(\mathcal{H})$ be with closed range such that $T^*\in\mathcal{P}\cdot\mathcal{L}^+$. Then, the $Q_T$-proper splitting of $T$ converges if and only if $\|P_{T^*}(I-T)\|<1$.
\item Let $T\in \mathcal{L}(\mathcal{H})$ be with closed range such that $T\in\mathcal{P}\cdot\mathcal{P}$. Then, the $Q_T$-proper splitting of $T$ converges.
\end{enumerate}
\end{cor}

\begin{proof}
1. It follows from Proposition \ref{Q-proper} and by \cite[Theorem 2.2]{AG-PB}.

2. It follows from Proposition \ref{Q-proper} and by \cite[Theorem 3.3]{AG-PA}.

3. If $T\in\mathcal{P}\cdot\mathcal{P}$ then  $T=P_TP_{T^*}$, see \cite[Theorem 3.1]{CM}. So that, $Q_T^\dagger V=P_{T^*}(I-P_TP_{T^*})$. Now, $\|P_{T^*}(I-P_TP_{T^*})\|=\|P_{T^*}P_{\mathcal{N}(T^*)}P_{T^*}\|<1$, where the inequality  holds by \cite[Lemma 10 and Theorem 12]{Deu} and  by \cite[Theorem 3.2]{CM}.
\end{proof}

\begin{pro}\label{group-proper}
Let  $T\in \mathcal{L}(\mathcal H)$ be  a split operator. If $\|P_{T^*}(I-T^2)\|<1$  then, the group proper splitting $T=T^\sharp -V$ converges. For the converse, if $P_{T^*}T^2\in\mathcal{L}^h$ and the group proper splitting of $T$ converges then $\|P_{T^*}(I-T^2)\|<1$.
\end{pro}

\begin{proof}
Note that $(T^\sharp)^\dagger V=P_{T^*}-(T^\sharp)^\dagger T= P_{T^*}-(T^\sharp)^\dagger  T^\sharp T T=P_{T^*}( I-T^2)$. Therefore, if $\|P_{T^*}(I-T^2)\|<1$ then $T=T^\sharp -V$ converges because $\rho((T^\sharp)^\dagger V)\leq\|P_{T^*}(I-T^2)\|<1$. Conversely, if $P_{T^*}T^2\in\mathcal{L}^h$ and $T=T^\sharp-V$ converges then $\|P_{T^*}(I-T^2)\|=\rho(P_{T^*}(I-T^2))=\rho((T^\sharp)^\dagger V)<1$. 
\end{proof}

We finish this section by studying different criteria  for the convergence of the $Q$-proper splitting and the group proper splitting of two operators related by means of the star and sharp orders. 

\begin{pro}
	Consider $S,T \in\mathcal{L}(\mathcal{H})$ split operators such that $\|P_{T^*}(I-T)\|<1$. The following assertions hold:
	\begin{enumerate}
		\item  If $S\overset{*}{\leq}T$ then the $Q_T$-proper splitting of $T$ and the $Q_S$-proper splitting of $S$ converge.
		\item If  $S\overset{\sharp}{\leq}T$ then the $Q_T$-proper splitting of $T$ and the $Q_S$-proper splitting of $S$ converge.
			\end{enumerate}
\end{pro}
\begin{proof}
	1. Note that  by Proposition \ref{Q-proper}, it holds that the $Q_T$-proper splitting of $T$ converges. Let $S=Q_S-W$ be the $Q_S$-proper splitting of $S$. Recall that $Q_S^\dagger W=P_{S^*}P_SW=P_{S^*}W=P_{S^*}(I-S)$. Since $S \overset{*}{\leq}T$ then $S=P_ST=TP_{S^*}$. Therefore, $\rho(Q_S^\dagger V)\leq \|P_{S^*}(I-S)\|= \|P_{S^*}(I-T)P_{S^*}\|= \|P_{S^*}P_{T^*}(I-T)P_{S^*}\|\leq \|P_{T^*}(I-T)\|<1$, where the last inequality follows by hypothesis. Hence, the $Q_S$-proper splitting of $S$ converges.
	
	2. Since $S\overset{\sharp}{\leq}T$, then $\mathcal{R}(S)\subseteq \mathcal{R}(T)$, $\mathcal{R}(S^*)\subseteq \mathcal{R}(T)$ and  $S=Q_ST=TQ_S$. Therefore,  $\rho(Q_S^\dagger W)\leq \|Q_S^\dagger W\|=\|Q_S^ \dagger Q_S(I-T)\| = \| P_{S^*}(I-T)\|= \| P_{S^*}P_{T^*}(I-T)\|\leq\| P_{T^*}(I-T)\|<1$, where the last inequality follows by hypothesis. Hence the $Q_S$-proper splitting of $S$ converges.
\end{proof}
\smallskip

\begin{pro}
Let $S,T \in \mathcal{L}(\mathcal H)$ be  split operators such that $\|P_{S^*}(I-ST)\|<1$. Then the following assertions hold:
\begin{enumerate}
\item If $S \overset{*}{\leq}T$  then the group proper splitting  of $S$ converges.
\item If $S \overset{\sharp}{\leq}T$    then the group proper splitting  of $S$ converges.
\end{enumerate}
\end{pro}

\begin{proof}
Let $T=T^\sharp-V$  and $S=S^\sharp-W$ the group proper splittings of $T$ and $S$, respectively.

1.  Since $S \overset{*}{\leq}T$ then $S=P_ST=TP_{S^*}$. Observe that $(S^\sharp)^\dagger W=P_{S^*}(1-S^2)=P_{S^*}(1-STP_{S^*})=P_{S^*}(1-ST)P_{S^*}$. Hence $\rho((S^\sharp)^\dagger W)\leq \|P_{S^*}(I-ST) \| <1, $ so that the group proper splitting of $S$ also converges.

2. Since $S \overset{\sharp}{\leq}T$ then $S=Q_ST=TQ_S$. Observe that $(S^\sharp)^\dagger W=(S^\sharp)^\dagger (S^\sharp-S)=P_{S^*}-(S^\sharp)^\dagger Q_ST=P_{S^*}-(S^\sharp)^\dagger S^\sharp ST=P_{S^*}(I-ST)$. Hence $\rho((S^\sharp)^\dagger W)\leq \|P_{S^*}(I-ST) \| <1, $ so that the group proper splitting of $S$ converges. 
\end{proof}

\section{Splittings for Hermitian operators}

In this section we introduce two new types of proper splittings for closed range Hermitian operators. 

\begin{defi}  Given $T\in \mathcal{L}^h$  with closed range we define the following splittings of $T$:
\begin{enumerate}
\item  the MP-proper splitting of $T$ is $T=T^\dagger-W$;
\item  the projection proper splitting of $T$ is $T=P_T-Z$. 
\end{enumerate}
\end{defi}

\begin{rem}
Note that for $T\in \mathcal L^h$, the MP-proper splitting and the projection splitting of $T$ are particular cases of the group proper splitting and the $Q_T$-proper splitting of $T$, respectively. In addition, if $T\in \mathcal L^+$ the projection splitting of $T$ coincides with the polar proper splitting of $T$. 
\end{rem}

\begin{rem}
Given  $T\in \mathcal{L}^h$ with closed range it could be natural to consider the proper splitting of $T$, $T=|T|-Y$. However, as in the finite dimensional case \cite{AG-splitting} it holds that if this proper splitting converges then $T\in\mathcal{L}^+$, so that $T=|T|$. In fact, since $T\in\mathcal{L}^h$ then $T\leq |T|$ and so that
 $Y\in\mathcal{L}^+$. Now, if the proper splitting $T=|T|-Y$ converges then $\rho(|T|^\dagger Y)=\rho(Y^{1/2}|T|^\dagger Y^{1/2})<1$. In addition it holds that $\mathcal{R}((|T|-T)^{1/2})\subseteq\overline{\mathcal{R}(|T|-T)}\subseteq\overline{\mathcal{R}(|T|)+\mathcal{R}(T)} =\overline{\mathcal{R}(T)+\mathcal{R}(T)}=\overline{\mathcal{R}(T)}=\mathcal{R}(T)$. Then, by Proposition \ref{BLT}, we get that $|T|-T\leq |T|$. Therefore, $T\in\mathcal{L}^+$.
\end{rem}

The following result gives a characterization for the convergence of the MP-proper splitting and the projection proper splitting.

\begin{pro}\label{convergencia MP y projection spl}
Consider $T\in \mathcal{L}^h$ a closed range operator. Then the following assertions hold:  
\begin{enumerate}
\item  the MP-proper splitting of $T$  converges if and only if $\|P_T-T^2\|<1$;
\item the projection proper splitting of $T$ converges if and only if $\|P_T-T\|<1$.
\end{enumerate}
\end{pro}

\begin{proof}
  Let   $T=T^\dagger-W$ the MP- proper splitting of $T$ and let $T=P_T-Z$ the projection proper splitting of $T$.

1. Since $TW=P_T-T^2\in  \mathcal{L}^h$ then $\rho(TW)=\|P_T-T^2\|$. Therefore, the assertion follows.

2.  Since $P_TZ=P_T(P_T-T)=P_T-T\in   \mathcal{L}^h$ then $\rho(P_TZ)=\|P_T-T\|$. Hence, the assertion follows.
\end{proof}

The next result is a comparison criterion between  the MP-proper splitting, the projection proper splitting and the polar proper  splitting  for Hermitian operators.

\begin{pro}\label{comp-PS,MPS,proS}
Consider $T\in \mathcal{L}^h$ a closed range operator and  $T=U_T-V=T^\dagger-W=P_T-Z$, the polar proper splitting, the MP-proper splitting and the projection proper splitting of $T$, respectively.  Then the following assertions hold:
\begin{enumerate}
	\item if  the projection proper splitting of $T$ converges then the polar proper splitting of $T$ converges. Moreover, $\rho(U_T^*V)\leq \rho(P_TZ)<1$.
	\item  if $\|T\|\leq 1$ and the MP-proper splitting of $T$ converges then the polar proper splitting of $T$ converges. Moreover, $\rho(U_T^*V)\leq \rho(TW)<1$.
	\item if  $\|P_T+T\|\leq 1$ and the projection proper splitting of $T$ converges then the MP-proper splitting of $T$ converges. Moreover, $\rho(TW)\leq \rho(P_TZ)<1$.
\end{enumerate}
\end{pro}
\begin{proof}
Let $T=U_T-V=T^\dagger-W=P_T-Z$, the polar proper splitting, the MP-proper splitting and the projection proper splitting of $T$, respectively.

	1. Suppose that the projection splitting of $T$ converges. Since $T\in\mathcal{L}^h$ then $T\leq |T|$. Therefore,  $P_TZ=P_T-T\geq P_T-|T|=U_TV$, so that $\rho(U_T^*V)\leq \rho(P_TZ)<1$.
	
	2. Note that $\rho(|T||T|^\dagger |T|)=\rho(|T|)=\||T|\|=\|T\|\leq 1$, where the second equality holds because $|T|\in \mathcal{L}^+$. Then, by Proposition \ref{BLT}, it holds that $|T|^2\leq |T|$. Now, observe that  $TW=P_T-|T|^2\geq P_T-|T|=U^*_TV$. Hence, $\rho(U_T^*V)\leq \rho(TW)<1$.
	
	3. Since $T\in\mathcal{L}^h$ then $P_T-T, \,P_T-T^2\in\mathcal{L}^h$. Then observe that, $\rho(TW)=\|P_T-T^2\|=\|(P_T-T)(P_T+T)\|\leq \|P_T-T\|=\rho(P_TZ)<1$.
\end{proof}

Given $T\in \mathcal{L}^h$ with closed range and $W\in\mathcal{L}(\mathcal{H})$, the equation $TX=W$ is solvable if and only if the equation $|T|X=W$ is solvable. Therefore, it is interesting to establish some relationships between  proper splittings of $T$ and  proper splittings of $|T|$.

\begin{pro}
Consider $T\in \mathcal{L}^h$ a closed range operator. Then the polar proper splitting of $T$ converges if and only if the projection proper splitting of $|T|$ converges. Moreover, if $T=U_T-V$ is the polar proper splitting of $T$ and $|T|=P_{|T|}-Z$ is the projection proper splitting of $|T|$ then $\rho(U_T^*V)=\rho(P_{|T|}Z)$.
\end{pro}

\begin{proof}
It is immediate.
\end{proof}

Next,  we build a proper splitting of an operator in $\mathcal{P}\cdot\mathcal{L}^h$ in terms of the projection proper splitting of an Hermitian associated factor. 

\begin{teo}\label{inducido en PLh}
Let $S\in\mathcal{P}\cdot\mathcal{L}^h$ be a split operator with closed range. Then $S=P_SP_{S^*}-W$ is a proper splitting of $S$. Moreover, the proper splitting $S=P_SP_{S^*}-W$  converges if and only if $\|P_{S^*}-Q^*S\|<1$, where $Q=Q_{\mathcal{R}(S)//\mathcal{N}(S)}$.
\end{teo}

\begin{proof}
Since $S\in\mathcal{P}\cdot\mathcal{L}^h$ is a split operator with closed range then, by \cite[Theorem 2.12 and Proposition 2.17]{AG-PB} there exist a unique $T\in\mathcal{L}^h$ with $\mathcal{N}(T)=\mathcal{N}(S)$ such that $S=P_ST$. Namely, $T=(S+S^*-P_S S^*)Q$, where $Q=Q_{\mathcal{R}(S)//\mathcal{N}(S)}$.  Let $T=P_T-V$ be the projection proper splitting of $T$. Let us see that $S=P_SP_T-P_SV$ is a proper splitting of $S$. In fact, $\mathcal{R}(P_SP_T)=P_S\mathcal{R}(T)=\mathcal{R}(P_ST)=\mathcal{R}(S)$. In addition, $\mathcal{N}(S)=\mathcal{N}(T)=\mathcal{R}(T)^\bot\subseteq \mathcal{N}(P_SP_T)$. Now, if $P_SP_Tx=0$ then $P_Tx\in\mathcal{R}(T)\cap\mathcal{R}(S)^\bot=\mathcal{N}(T)^\bot\cap\mathcal{R}(S)^\bot=\mathcal{N}(S)^\bot \cap\mathcal{R}(S)^\bot=(\mathcal{N}(S)+\mathcal{R}(S))^\bot=\{0\}$, so that $x\in\mathcal{R}(T)^\bot=\mathcal{N}(S)$. Then, $\mathcal{N}(P_SP_T)=\mathcal{N}(S)$ and therefore $S=P_SP_T-P_SV$ is a proper splitting of $S$.  Finally, since $(P_SP_{S^*})^\dagger W=Q^*(P_SP_{S^*}-S)=Q^*(P_SP_{S^*}-P_ST)=P_{S^*}-T\in\mathcal{L}^h$ then the proper splitting $S=P_SP_{S^*}-W$ converges if and only if $\|P_{S^*}-T\|<1$. The assertion follows replacing $T$ by $(S+S^*-P_S S^*)Q$, where $Q=Q_{\mathcal{R}(S)//\mathcal{N}(S)}$.
\end{proof}

\begin{rem}
With the same notation that  in the proof of Theorem \ref{inducido en PLh} it holds that  the projection proper splitting of $T$ converges if and only if  the proper splitting $S=P_SP_T-P_SV$ converges. In fact, by \cite[Theorem 4.1]{MR2653816}  we get that $(P_SP_T)^\dagger P_SV=Q^*P_SV=Q^*V=Q^*(P_T-T)=P_T-T$, because $\mathcal{R}(T)=\mathcal{R}(S^*)$. Then $\rho((P_SP_T)^\dagger P_SV)=\|P_T-T\|$ and the assertion follows from Proposition \ref{convergencia MP y projection spl}.
\end{rem}

\begin{cor}\label{inducido-PA}
Consider $S\in\mathcal{P}\cdot\mathcal{L}^+$ with closed range. Then there exists $T\in \mathcal{L}^+$ with $\mathcal{N}(T)=\mathcal{N}(S)$ such that $S=P_ST$. Moreover, the projection proper splitting of $T$, $T=P_T-V$, induces a proper splitting of $S$, namely, $S=P_SP_T-P_SV$. In addition, the projection proper splitting of $T$ converges if and only if $S=P_SP_T-P_SV$ converges.
\end{cor}

\begin{proof}
The existence of $T\in\mathcal{L}^+$ with $\mathcal{N}(T)=\mathcal{S}$ such that $S=P_ST$ is  guarantee by \cite[Proposition 4.1]{AG-PA}. Then the proof follows as in the above theorem.
\end{proof}

Consider $S\in\mathcal{L}(\mathcal{H})$ with closed range such that $S=P_ST$, where $T\in\mathcal{L}^+$ and $\mathcal{N}(T)=\mathcal{N}(S)$. The next result shows that the proper splitting of $S$ induced by the projection proper splitting of  $T$  in Corollary \ref{inducido-PA}, converges faster than the polar proper splitting of $S$.

\begin{cor}\label{inducido en PA para S contraccion}
Consider $S\in\mathcal{L}(\cH)$ with closed range. Suppose that  there exists $T\in \mathcal{L}^+$ with $\|T\|\leq 1$, $\mathcal{N}(T)=\mathcal{N}(S)$ and $\mathcal{R}(T)$ closed such that $S=P_ST$. Consider $T=P_T-V$ the projection proper splitting of $T$ and  $S=P_SP_T-P_SV=U_S-W$ the  proper splitting of $S$ induced by  $T$  and the polar proper splitting of $S$, respectively. Then it holds that $\rho((P_SP_T)^\dagger P_SV)\leq \rho(U_S^*W)$. 
\end{cor}

\begin{proof}
Since $S=P_ST$ with $T\in\mathcal{L}^+$ and $\mathcal{N}(T)=\mathcal{N}(S)$ then $\mathcal{R}(T)=\mathcal{R}(S^*)$. Therefore $(P_SP_T)^\dagger P_SV=P_{S^*}-T\geq 0$ by Proposition \ref{BLT}.  Now, note that $|S|=(TP_ST)^{1/2}\leq T$. Then $P_{S^*}-|S|\geq P_{S^*}-T\geq 0$ and so $ \rho(U_S^*W)=\|P_{S^*}-|S|\|\geq \|P_{S^*}-T\|=\rho((P_SP_T)^\dagger P_SV)$. Hence the assertion follows.
\end{proof}

\begin{rem}
Let us see that in Corollary \ref{inducido en PA para S contraccion} the inequality can be strict.  Take $S=P_ST\in\mathcal{P}\cdot\mathcal{L}^+$, where $P_S=\left(\begin{matrix}
 1/2& 1/2 \\
1/2 & 1/2
\end{matrix}\right)$ and  $T=\left(\begin{matrix}
 1/4& 0 \\
0 & 0
\end{matrix}\right)$. It holds that $T\in\mathcal{L}^+$ and $\mathcal{N}(S)=\mathcal{N}(T)$.  The projection proper splitting  of $T$ is $T=P_T-V= \left(\begin{matrix}
 1& 0 \\
0 & 0
\end{matrix}\right) - \left(\begin{matrix}
 3/4& 0 \\
0 & 0
\end{matrix}\right) $.  Then the proper splitting of $S$ induced by the projection proper splitting of $T$ is $S=P_SP_T-P_SV=\left(\begin{matrix}
 1/2& 1/2 \\
1/2 & 1/2
\end{matrix}\right) \left(\begin{matrix}
 1& 0 \\
0& 0
\end{matrix}\right) - \left(\begin{matrix}
 1/2& 1/2 \\
1/2 & 1/2
\end{matrix}\right) \left(\begin{matrix}
 3/4& 0 \\
0& 0
\end{matrix}\right) $. 

Now, $(P_SP_T)^\dagger P_SV=\left(\begin{matrix}
 1/2& 0 \\
1/2 & 0
\end{matrix}\right)^\dagger \left(\begin{matrix}
 1/2& 1/2 \\
1/2 & 1/2
\end{matrix}\right) \left(\begin{matrix}
 3/4& 0 \\
0& 0
\end{matrix}\right) =\left(\begin{matrix}
 1& 1 \\
0& 0
\end{matrix}\right) \left(\begin{matrix}
 1/2& 1/2 \\
1/2 & 1/2
\end{matrix}\right) \left(\begin{matrix}
 3/4& 0 \\
0& 0
\end{matrix}\right) =\left(\begin{matrix}
 3/4& 0 \\
0& 0
\end{matrix}\right)$. Then $\rho((P_SP_T)^\dagger P_SV)=3/4$. 
On the other hand, the polar proper splitting of $S$ is $S=U_S-W=\left(\begin{matrix}
\sqrt{2}/2 & 0 \\
\sqrt{2}/2 & 0
\end{matrix}\right)-\left(\begin{matrix}
\sqrt{2}/2 -1/8& 0 \\
\sqrt{2}/2-1/8 & 0
\end{matrix}\right)$, so that, $U_S^*W=\left(\begin{matrix}
\sqrt{2}/2 & \sqrt{2}/2 \\
0 & 0
\end{matrix}\right) \left(\begin{matrix}
\sqrt{2}/2 -1/8& 0 \\
\sqrt{2}/2-1/8 & 0
\end{matrix}\right)=\left(\begin{matrix}
(8-\sqrt{2})/8& 0 \\
0 & 0
\end{matrix}\right)$. Hence, $\rho(U_S^*W)=1-\sqrt{2}/8$  and therefore $\rho((P_SP_T)^\dagger P_SV)<\rho(U_S^*W)$.
\end{rem}

\begin{pro}\label{acotados por Hermitiano1}
Let $S\in\mathcal{L}(\mathcal{H})$ be a split operator. If $T\in\mathcal{L}^h$ then the following assertions hold:
\begin{enumerate}
\item If $S\overset{*}{\leq} T$ and the projection proper splitting of $T$ converges then the $Q_S$-proper splitting of $S$ converges.
\item If $S\overset{\sharp}{\leq} T$ and the projection proper splitting of $T$ converges then the $Q_S$-proper splitting of $S$ converges.
\end{enumerate}
\end{pro}
\begin{proof}
	Consider $T=P_T-Z$ the projection splitting of $T$ and $S=Q_S-W$  the $Q_S$-proper splitting of $S$:
	
	1. Since $S\overset{*}{\leq} T$ and $T\in\mathcal{L}^h$ then $\mathcal{R}(S)\subseteq \mathcal{R}(T)$, $\mathcal{R}(S^*)\subseteq \mathcal{R}(T)$ and $S=P_ST=TP_{S^*}$. Observe that $\rho(Q_S^\dagger W)\leq \| Q_S^\dagger W\|=\|P_{S^*}(1-S)\|=\|P_{S^*}P_{T}(1-T)P_{S^*}\|\leq \|P_{T}(I-T)\|$. Since the projection splitting of $T$ converges then, by Proposition \ref{convergencia MP y projection spl}, it holds that $\|P_{T}(I-T)\|<1$. Therefore, the $Q_S$-proper splitting of $S$ converges.
	
	2. It is similar to the proof of item 1.
\end{proof}

The next result allows to give an extension of the above proposition for the case of the minus order as we will see in Remark \ref{extension a minus order}. 

\begin{pro}\label{minus para projection de T y QS de S}
Let $S\in\mathcal{L}(\mathcal{H})$ be a split operator and $T\in\mathcal{L}^h$ such that  $\mathcal{R}(S^*)\subseteq \mathcal{R}(T)$ and $S^*=QT$, for some $Q\in\mathcal{Q}$ with $\mathcal{R}(Q)=\mathcal{R}(S^*)$ and $\|Q\|<\frac{1}{\|P_T-T\|}$. If the projection proper splitting of $T$ converges then the $Q_S$-proper splitting of $S$ converges.
\end{pro}

\begin{proof}
  Let $T\in\mathcal{L}^h$ such that  $\mathcal{R}(S^*)\subseteq \mathcal{R}(T)$  and $S^*=QT$, where $Q\in\mathcal{Q}$ and $\mathcal{R}(Q)=\mathcal{R}(S^*)$. The $Q_S$ proper splitting of $S$ is $S=Q_S-W$, where $Q_S=Q_{\mathcal{R}(S)//\mathcal{N}(S)}$. Then $Q_S^\dagger W=P_{S^*}P_S(Q_S-S)=P_S^*(I-S)$. Now, since $\mathcal{R}(S^*)\subseteq \mathcal{R}(T)$ we get that $\|P_{S^*}(I-S)\|=\|(I-S^*)P_{S^*}\|=\|(I-QT)P_{S^*}\|=\|(Q-QT)P_{S^*}\|=\|Q(I-T)P_{S^*}\|=\|QP_T(I-P_T)P_{S^*}\|\leq \|Q\| \|P_T-T\|<1$.
\end{proof}

\begin{rem}\label{extension a minus order}
Suppose  that $S\in\mathcal{L}(\mathcal{H})$ is a split operator and $T\in\mathcal{L}^h$ are such that $S\overset{-}\leq T$. This means that $S=Q_1T$ and $S^*=Q_2T$, for some $Q_1,Q_2\in\mathcal{Q}$ with $\mathcal{R}(Q_1)=\mathcal{R}(S)$, $\mathcal{R}(Q_2)=\mathcal{R}(S^*)$. Observe that  if  $\|Q_2\|<\frac{1}{\|P_T-T\|}$ then we are in the case of  proposition \ref{minus para projection de T y QS de S}. 
\end{rem}

\begin{pro}\label{acotacion para MP}
Consider $S,T\in\mathcal{L}^h$   closed range operators  such that $S\overset{*}{\leq} T$. If the MP-proper splitting of $T$ converges then the MP-proper splitting of $S$ converges.
\end{pro}

\begin{proof}
Since $S\overset{*}{\leq} T$ then $S=P_ST=TP_S$ and $\mathcal{R}(S)\subseteq \mathcal{R}(T)$ . If the MP-splitting of $T$ converges then, by Proposition \ref{convergencia MP y projection spl},  it holds that $\|P_T-T^2\|<1$. Now, $ \|P_S-S^2\|=\|P_S-P_S T^2P_S\|=\|P_S(I-T^2)P_S\|=\|P_SP_T(I-T^2)P_S\|\leq \|P_T(I-T^2)\|<1$.  Therefore the assertion follows by Proposition \ref{convergencia MP y projection spl}.
	
\end{proof}

\subsection{An application to symmetric approximations of frames}

In this section we will apply the results obtained to find the synthesis operator of the symmetric approximation of a frame in a Hilbert space. First, we state  a consequence of Proposition \ref{polarcompacto} for the  polar proper splitting of $T$  that  will be useful in this subsection: 

\begin{pro}
Consider $T\in\mathcal{L}(\mathcal{H})$ with closed range. If $T=U_T-V$ is the polar proper splitting of $T$ then  the following statements are equivalent:
\begin{enumerate}
\item $U_T^*V\in\mathcal{K}$
\item $|T|^\dagger-P_{T^*}\in\mathcal{K}$;
\item $|T|-P_{T^*}\in\mathcal{K}$;
\item $U_T-T\in\mathcal{K}$.
\end{enumerate}
\end{pro}

\begin{proof}

$1\leftrightarrow 4.$ It follows from Proposition \ref{polarcompacto}.

$2\leftrightarrow 3.$ It follows from the fact that $|T|-P_{T^*}=|T|(P_{T^*}-|T|^\dagger)$ and $|T|^\dagger-P_{T^*}=|T|^\dagger(P_{T^*}-|T|)$.

$3 \leftrightarrow 4.$ It follows from the fact  that $(U_T-T)=U_T(P_{T^*}-|T|)$ and $P_{T^*}-|T|=U_T^*(U_T-T)$.
\end{proof}

A sequence $\{f_i\}_{i\in\mathbb{N}}\subseteq \mathcal{H}$ is a frame for a subspace $\mathcal{S}\subseteq\mathcal{H}$ if there exist constants $A,B>0$ such that 
$$
A \|x\|^2 \leq \sum\limits_{i=i}^{\infty}|\langle{x,f_i}\rangle|^2\leq B\|x\|^2,
$$
for all $x\in\mathcal{S}$. If $A=B$ the frame is called tight and if $A=B=1$ the frame is said normalized tight. Two frames $\{f_i\}_{i\in\mathbb{N}}$ and $\{g_i\}_{i\in\mathbb{N}}$ for subspaces $\mathcal{S}$ and $\mathcal{T}$ respectively, are called weakly similar if  there exists a bounded linear invertible operator $T:\mathcal{S}\rightarrow \mathcal{T}$ such that $T(f_i)=g_i$ for all $i\in\mathbb{N}$.

A normalized tight frame $\{v_i\}_{i\in\mathbb{N}}$ in $\mathcal{T}\subseteq \mathcal{H}$  is said to be a symmetric approximation of $\{f_i\}_{i\in\mathbb{N}}$ if the inequality
$$
\sum\limits_{i=1}^{\infty}\|u_i-f_i\|^2 \geq \sum\limits_{i=1}^{\infty}\|v_i-f_i\|^2
$$
holds for all normalized tight frames $\{u_i\}_{i\in\mathbb{N}}$ which are weakly similar to $\{f_i\}_{i\in\mathbb{N}}$ and the sum of the right side is finite.

Given a frame $\{f_i\}_{i\in\mathbb{N}}\subseteq \mathcal{H}$, the operator  $F:\ell^2\rightarrow \mathcal{H}$ defined by $F(e_i)=f_i$ is the synthesis operator associated to $\{f_i\}$, where $\{e_i\}$ is the standart orthonormal basis of $\ell^2$.  It is known \cite[Theorem 2.3]{FPT} that given  a frame $\{f_i\}_{i\in\mathbb{N}}\subseteq \mathcal{H}$ there exists a symmetric approximation of it if $P_{F^*}-|F|\in\mathcal{L}^2$, where $\mathcal{L}^2$ is the Hilbert-Schmidt class of $\mathcal{L}(\cH)$. Moreover, the symmetric approximation is given by the frame $\{U_F e_i\}_{i\in\mathbb{N}}$, where again $\{e_i\}$ is the canonical orthonormal basis of $\ell^2$ and $U_F$ is the partial isometry of the polar decomposition of $F$, $F=U_F|F|$.
Now, if we consider the equation $|F|X=F^*$ it is well known that $U_F^*$ is its reduced Douglas solution. Then, we can consider the projection proper splitting  $|F|=P_{F^*}-V$ of $|F|$ (observe that this proper splitting is the polar proper splitting of $|F|$) to obtain the operator $U_F^*$.  We summarize in the following result this fact:

\begin{pro}
Consider  $\{f_i\}_{i\in\mathbb{N}}$ a frame for a subspace $\mathcal{S}\subseteq \mathcal{H}$ and $|F|=P_{F^*}-V$  the projection proper splitting of $|F|$. If  $P_{F^*}-|F|\in\mathcal{L}^2$ and $\|P_{F^*}-|F|\|<1$ then the iterative process 
 $$
 X^{i+1}=(P_{F^*}-|F|)X^i+F^*,
 $$
converges to the adjoint  of the synthesis operator of the symmetric approximation of  $\{f_i\}_{i\in\mathbb{N}}$.
\end{pro}

\begin{proof}
It follows from \cite[Theorem 2.3]{FPT}, Theorem \ref{convergencia} and Proposition \ref{convergencia MP y projection spl}.
\end{proof}

\begin{rem}
Observe that from Theorem \ref{condiciones suficientes de convergencia} if $P_{F^*}-|F|\in \mathcal{L}^+\cap\mathcal{L}^2$ then the projection proper splitting $|F|=P_{F^*}-V$ of  $|F|$, converges. So that, $\|P_{F^*}-|F|\|<1$.  However, condition $\|P_{F^*}-|F|\|<1$ does not imply $P_{F^*}-|F|\in\mathcal{L}^+$, in general. Therefore, with this particular proper splitting we get a weaker condition that the one given in Theorem \ref{condiciones suficientes de convergencia} to guarantee the convergence. 
\end{rem}

\section{A few more on   induced splittings}

In this section we study  induced splittings for the case of two operators related by an invertible operator.

\begin{pro}
Consider $S\in \mathcal{L}(\mathcal{H})$ and $T\in\mathcal{L}^h$ such that $S=TG$ for some $G\in\mathcal{G}$. Then the projection proper splitting $T=P_T-V$ of $T$ induces a proper splitting of $S$, namely, $S=P_TG-VG$.  Moreover,  the projection proper splitting of $T$ converges if and only if the $S=P_TG-VG$ converges.
\end{pro}

\begin{proof}
Observe that $S=TG=P_TG-VG$. Let us see that $\mathcal{R}(P_TG)=\mathcal{R}(S)$ and $\mathcal{N}(P_TG)=\mathcal{N}(S)$. In fact, $\mathcal{R}(P_TG)=\mathcal{R}(P_T)=\mathcal{R}(T)=\mathcal{R}(S)$. In addition, since  $P_TG=T^\dagger TG$ we get that  $\mathcal{N}(S)=\mathcal{N}(TG)\subseteq \mathcal{N}(P_TG)$. Now, if $0=P_TGx=T^\dagger TGx$ then  $TGx\in\mathcal{R}(T)\cap\mathcal{R}(T)^\bot=\{0\}$, so that $x\in\mathcal{N}(TG)=\mathcal{N}(S)$. Thus, $\mathcal{N}(P_TG)=\mathcal{N}(S)$. Then $S=P_TG-VG$ is a proper splitting of $S$. Observe that $P_TG=T^\dagger TG=T^\dagger S$. Finally, $(T^\dagger S)^\dagger VG=S^\dagger TVG=S^\dagger T(P_T-T)G=S^\dagger (TG-T^2G)=S^\dagger (S-TS)=S^\dagger (P_T-T)S$, where the first equality holds by Proposition \ref{dagger de AB con rg iguales}. Now, as $\sigma((T^\dagger S)^\dagger VG)\cup\{0\}=\sigma{(S^\dagger (P_T-T)S)}\cup\{0\}=\sigma((P_T-T)SS^\dagger)\cup\{0\}=\sigma((P_T-T)P_T)\cup\{0\}=\sigma((P_T-T))\cup\{0\}$ then we get that $\rho((T^\dagger S)^\dagger VG)=\rho({P_T-T})$, and the assertion follows.
\end{proof}

\begin{pro}\label{splitting inducido}
Consider $S,T\in\mathcal{L}(\mathcal{H})$ such that $S=TW$, for some $W\in\mathcal{U}$. If $T=U-V$ is a proper splitting of $T$ then $S=UW-VW$ is a proper splitting of $S$. Moreover, it holds that the proper splitting of $T$ converges if and only if the proper splitting $S=UW-VW$  converges. 
\end{pro}

\begin{proof}
Let us see that $S=UW-VW$ is a proper splitting of $S$. Note that $\mathcal{R}(UW)=\mathcal{R}(U)=\mathcal{R}(T)=\mathcal{R}(S)$. In addition, $\mathcal{N}(UW)=\mathcal{N}(TW)=\mathcal{N}(S)$. Indeed, $UWx=0$ if and only if $Wx\in\mathcal{N}(U)=\mathcal{N}(T)$ if and only if $TWx=0$. Therefore $S=UW-VW$ is a proper splitting of $S$. Now since $\mathcal{R}(U^*UW)\subseteq\mathcal{R}(W)=\mathcal{H}$ and $\mathcal{R}(WW^*U^*)=\mathcal{R}(U^*)$ then, by Proposition \ref{dagger de AB con rg iguales}  $(UW)^\dagger=W^*U^\dagger$.  Finally since $(UW)^\dagger VW=W^*U^\dagger V W$ then $\rho((UW)^\dagger VW)=\rho(U^\dagger V)$. Then the assertion follows.
\end{proof}
 
\begin{cor} \label{coro splitting inducido}
Let $S,T\in\mathcal{L}(\mathcal{H})$ such that $S=TW$, for some $W\in\mathcal{U}$.  Then the following assertions hold:
\begin{enumerate}
\item If $T=U_T-V$ is the polar proper splitting of $T$ then the induced proper splitting of $S$, $S=U_TW-VW$ is the polar proper splitting of $S$.
\item If $S,T\in\mathcal{L}^h$ and $T=T^\dagger -V$ is the MP-proper splitting of $T$ then the induced proper splitting of $S$, $S=T^\dagger W-VW$ is the MP-proper splitting of $S$.
\end{enumerate}
\end{cor}

\begin{proof}

1. By Proposition \ref{splitting inducido} it is sufficient to note that $U_TW=U_S$. If fact, $\mathcal{R}(U_TW)=\mathcal{R}(S)$, $\mathcal{N}(U_TW)=\mathcal{N}(S)$ and $U_TWW^*U_T^*=U_T^*U_T=P_T=P_S$. So that, $U_TW=U_S$.

2. By Proposition \ref{splitting inducido} it is sufficient to note that $T^\dagger W=S^\dagger$. Since $S=S^*$, by Proposition \ref{dagger de AB con rg iguales} it holds that $S^\dagger=(W^*T)^\dagger=T^\dagger W$. So, the assertion follows.
\end{proof}

\begin{cor}
Let $S,T\in\mathcal{L}(\mathcal{H})$ such that $S=TW$, for some $W\in\mathcal{U}$.  Then the following assertions hold:
\begin{enumerate}
\item  The polar proper splitting of $T$ converges if and only if the polar proper splitting of $S$ converges.
\item  If $S,T\in\mathcal{L}^h$ then the MP-proper splitting of $T$ converges if and only if the MP-proper splitting of $S$ converges.
\end{enumerate}
\end{cor}

\begin{proof} 
It follows from Propositions \ref{splitting inducido} and Corollary \ref{coro splitting inducido}.
\end{proof}

The last result shows that the polar proper splitting of  all operators in the unitary orbit of an  $T\in\mathcal{L}(\mathcal{H})$ have the same behavior with respect to the convergence. Moreover, the speed of convergence of  these  polar proper splittings coincides. 

\begin{pro}
Consider $S,T\in\mathcal{L}(\mathcal H)$ such that $S=UTU^*$, where $U\in\mathcal{U}$. Then, the polar proper splitting of $S$ converges if and only if the polar proper splitting of $T$ converges. Moreover, if $S=U_S-W$ and $T=U_T-V$ are the polar proper splittings of $S$ and $T$, respectively,  then $\rho(U_S^*W)=\rho(U_T^*V)$.
\end{pro}

\begin{proof}
Since $S=UTU^*$ then $U^*SU=T=U^*U_SU-U^*WU$. Let us see that $U^*U_SU=U_T$. In fact, $|S|=U|T|U^*$ then $|S|^\dagger=U|T|^\dagger U^*$. Now, $U_S=S|S|^\dagger=UTU^*U|T|^\dagger U^*=UT|T|^\dagger U^*=UU_TU^*$. So that $U^*U_SU=U_T$ and then $U^*WU=V$.  Therefore, $\rho(U^*_TV)=\|U^*_TV\|=\|U^*U^*_SUV\|=\|UU^*U^*_SUVU^*\|=\|U^*_SUVU^*\|=\|U^*_SW\|=\rho(U_S^*W)$.
\end{proof}


\begin{thebibliography}{10}

\bibitem{AntCanSto2010}
J. Antezana, C. Cano, I. Mosconi, D.~Stojanoff; \emph{A note on the star
  order in {H}ilbert spaces}, Linear Multilinear Algebra \textbf{58} (2010),
  no.~7-8, 1037--1051.

  \bibitem{arias2008generalized}
M.L.~Arias, G.~Corach, M.C.~Gonzalez; \emph{Generalized inverses and Douglas
  equations}, Proc. Amer. Math. Soc. \textbf{136}
  (2008), no.~9, 3177--3183.

\bibitem{AG-PA} M.L. Arias, G. Corach, M.C. Gonzalez; \emph{Products of projections and positive operators}, Linear Algebra Appl. \textbf{439} (2013), 1730-1741.

 \bibitem{AG-PB} M.L. Arias,  M.C. Gonzalez; \emph{Products of projections and self-adjoint operators}, Linear Algebra Appl. \textbf{555} (2018), 70-83.

  \bibitem{AG-splitting} M.L. Arias, M.C. Gonzalez; \emph{Proper splittings and reduced solutions of matrix equations}, J. Math. Anal. Appl. \textbf{505} (2022), no 1, doi.org/10.1016/j.jmaa.2021.125588 .
  
  \bibitem{MR1032688}
J.K. Baksalary, E.P. Liski, G. Trenkler; \emph{Mean square
  error matrix improvements and admissibility of linear estimators}, J.
  Statist. Plann. Inference \textbf{23} (1989), no.~3, 313--325. 
  
  \bibitem{MR348984}
A. Berman,  R.J. Plemmons;  \emph{Cones and iterative methods for
  best least squares solutions of linear systems}, SIAM J. Numer. Anal.
  \textbf{11} (1974), 145--154. 
  
  \bibitem{MR317088}
R. Bouldin;  \emph{The pseudo-inverse of a product}, SIAM J. Appl. Math.
  \textbf{24} (1973), 489--495.
  
   \bibitem{Chium2021Onaconjeture} E. Chiumiento; \emph{On a conjecture by Mbekhta about best approximation by polar factors}, Proc. Amer. Math. Soc. \textbf{149} (09) (2021): 3913--3922.
   
   \bibitem{MR1628383}
J.J. Climent, C. Perea; \emph{Some comparison theorems for weak
  nonnegative splittings of bounded operators}, Proceedings of the {S}ixth
  {C}onference of the {I}nternational {L}inear {A}lgebra {S}ociety ({C}hemnitz,
  1996) \textbf{275/276} (1998), 77--106. 
   
   \bibitem{MR1695402}
J.J. Climent, C. Perea; \emph{Convergence and comparison theorems for multisplittings},
  \textbf{6} (1999), Czech-US Workshop in Iterative Methods and Parallel Computing,
  Part 2 (Milovy, 1997), 93--107. 
  
  \bibitem{MR2653816}
G. Corach, A. Maestripieri; \emph{Polar decomposition of oblique
  projections}, Linear Algebra Appl. \textbf{433} (2010), no.~3, 511--519.
  
  \bibitem{CM} G. Corach, A. Maestripieri; \emph{Products of orthogonal projections and polar decompositions}, 
Linear Algebra Appl. \textbf{434} (2011), 1594-1609.
   
   \bibitem{MR745083}
G. Csordas,  R.S. Varga; \emph{Comparisons of regular splittings of
  matrices}, Numer. Math. \textbf{44} (1984), no.~1, 23--35. 
  
 
 \bibitem{Deu} F. Deutsch; \emph{The angles between subspaces in Hilbert space}, in: S.P. Singh (Ed.), Approximation Theory, Wavelets and Applications, Kluwer, Netherlands
(1995), 107-130.
  
  \bibitem{MR0203464}
  R.G. Douglas; \emph{On majorization, factorization, and range inclusion of
  	operators on {H}ilbert space}, Proc. Amer. Math. Soc. \textbf{17} (1966),
  413--415. 
  
  \bibitem{Dji2017Theminusordes}
  M.S.~Djiki\'{c}, G.~Fongi,  A.~Maestripieri; \emph{The minus order and range additivity}, Linear Algebra Appl. \textbf{531} (2017),
  234--256. 
  
  
  \bibitem{2001DjorStani} D.S. Djordjevic, P.S. Stanimirovic; \emph{Splittings of operators and generalized inverses}, Publ. Math. Debrecen \textbf{59} 1-2 (2001), 147--159.
 
\bibitem{Efimov}
   M.A. Efimov; \emph{On the $\overset{\sharp}\leq$-order on the set of linear bounded operators in Banach Space}, Math. Notes no \textbf{93}(5) (2013), 784-788.
  
  \bibitem{MR1969059}
L. Elsner, A. Frommer, R. Nabben, H. Schneider, D.B.
  Szyld;  \emph{Conditions for strict inequality in comparisons of spectral
  radii of splittings of different matrices}, \textbf{363}, 2003, Special issue on
  nonnegative matrices, $M$-matrices and their generalizations (Oberwolfach,
  2000), ~65--80. 
  
  
 \bibitem{fongi2023moore}
 G. Fongi, M.C. Gonzalez; \emph{Moore-Penrose inverse and partial orders on Hilbert space operators}, Linear Algebra Appl. \textbf{674} (2023), 1-20.

\bibitem{FPT}
M. Frank, V. Paulsen, T. Tiballi; \emph{Symmetric approximation of frames and bases in Hilbert spaces}, Trans. Am. Math. Soc. \textbf{354} (2002), 777-793.

\bibitem{greville1966note}
T.N.E. Greville;  \emph{Note on the generalized inverse of a matrix
  product}, SIAM  Rev.  \textbf{8} (1966), no.~4, 518--521.
  
  \bibitem{Gro}
  C.W. Groetsch; \emph{Generalized inverses of linear operators, Representation and Approximation}, Pure and Applied Mathematics \textbf{37},  Dekker, New York (1977).
  
  \bibitem{MR0075677}
A.S. Householder;  \emph{On the convergence of matrix iterations}, Rep.
  ORNL 1883, Oak Ridge National Laboratory, Oak Ridge, Tenn., 1955.

\bibitem{MR651705}
S. Izumino;  \emph{The product of operators with closed range and an
  extension of the reverse order law}, Tohoku Math. J. (2) \textbf{34} (1982),
  no.~1, 43--52. 

\bibitem{John}
F.~John; \emph{Advanced numerical analysis}, Lecture Notes, Dept. of
  Mathematics, New York Univ, 1956.

 \bibitem{jose2015partial}
 S.~Jose,  K.C. Sivakumar; \emph{On partial orders of hilbert space operators}, Linear Multilinear Algebra \textbf{63} (2015), no.~7, 1423--1441.
  
  \bibitem{MR3671533}
M.R. Kannan;  \emph{P-proper splittings}, Aequationes Math. \textbf{91}
  (2017), no.~4, 619--633. 
  
 \bibitem{LiuHuang}
   X. Liu, S. Huang;  \emph{Proper splitting for the generalized inverse $A^{(2)}_{T,S}$ and its application on Banach spaces}, Abstr. Appl. Anal. (2012),  Article ID 736929,  9 pp.
  
  \bibitem{MR3141710}
D. Mishra;  \emph{Nonnegative splittings for rectangular matrices},
  Comput. Math. Appl. \textbf{67} (2014), no.~1, 136--144. 
  
 \bibitem{2009Nacevska} B. Nacevska;  \emph{Iterative methods for computing generalized inverses and splittings of operators},  Appl. Math. Comp. \textbf{208} 1 (2009), 186--188.
  
    
  \bibitem{Robert}
  P. Robert;  \emph{On the group-inverse of a linear transformation}, J. Math. Anal. Appl. \textbf{22} (1968), 658--669.
  
  
   \bibitem{Seb} Z. Sebesty\'en; \emph{Restrictions of positive operators}, Acta Sci.Math. \textbf{46} (1983), 299-301.
   
  \bibitem{MR1113154}
Y. Song; \emph{Comparisons of nonnegative splittings of matrices},
  Linear Algebra Appl. \textbf{154/156} (1991), 433--455. 
  
  \bibitem{MR1930390}
Y. Song;  \emph{Comparison theorems for splittings of matrices}, Numer.
  Math. \textbf{92} (2002), no.~3, 563--591.
  
  
   \bibitem{varga}
R.S. Varga;  \emph{Matrix iterative analysis}, Prentice-Hall, Inc.,
  Englewood Cliffs, N.J., 1962.
  
  \bibitem{MR1286436}
Z.I. Wo\'{z}nicki;  \emph{Nonnegative splitting theory}, Japan J. Indust.
  Appl. Math. \textbf{11} (1994), no.~2, 289--342. 
\end{thebibliography}
\end{document}